\documentclass[11pt]{amsart}
\usepackage{amssymb, amsmath, amsthm,ulem}
\usepackage[latin1]{inputenc}
\usepackage{graphicx}

\usepackage[hidelinks]{hyperref}

\usepackage{color}
\usepackage{ulem}
\usepackage{epstopdf}
\usepackage{cancel}
\usepackage{tikz}

\usepackage[a4paper, centering]{geometry}

\usepackage{color}
\usepackage{graphicx}

\usepackage{scalerel,stackengine}

\usepackage{mathtools}

\newtheorem{theorem}{Theorem}

\newtheorem{lemma}[theorem]{Lemma}
\newtheorem{corollary}[theorem]{Corollary}
\theoremstyle{remark}

\theoremstyle{definition}
\newtheorem{definition}[theorem]{Definition}
\theoremstyle{remark}

\definecolor{verde}{RGB}{20,150,100}
\definecolor{purple}{RGB}{200,30,200}

\newcommand{\EEE}{\color{black}}

\stackMath
\newcommand\reallywidecheck[1]{%
\savestack{\tmpbox}{\stretchto{%
  \scaleto{%
    \scalerel*[\widthof{\ensuremath{#1}}]{\kern-.6pt\bigwedge\kern-.6pt}%
    {\rule[-\textheight/2]{1ex}{\textheight}}%
  }{\textheight}%
}{0.5ex}}%
\stackon[1pt]{#1}{\scalebox{-1}{\tmpbox}}%
}

\def\R{\mathbb{R}}
\def\C{\mathcal{C}}
\def\N{\mathbb{N}}

\newcommand{\Om}{\Omega}

\newcommand{\sq}{\subseteq}

\def \e{\varepsilon}

\def\ds{\displaystyle}

\begin{document}
\title[]{ 
Concavity  and hot spots in elliptic problems \\  under mixed boundary conditions }

\bigskip\bigskip

\vfill\eject 
\author[]{Dorin Bucur, Ilaria Fragal\`a}

\thanks{}

\address[Dorin Bucur]{
Universit\'e  Savoie Mont Blanc, Laboratoire de Math\'ematiques CNRS UMR 5127 \\
  Campus Scientifique \\
73376 Le-Bourget-Du-Lac (France)
}
\email{dorin.bucur@univ-savoie.fr}

\address[Ilaria Fragal\`a]{
Dipartimento di Matematica \\ Politecnico  di Milano \\
Piazza Leonardo da Vinci, 32 \\
20133 Milano (Italy)
}
\email{ilaria.fragala@polimi.it}


\makeatletter
\def\@setsubjclass{%
  \itshape 2020 Mathematics Subject Classification.\enspace
  \upshape\@subjclass\@addpunct.}
\makeatother\date{\today}

\begin{abstract}  We consider the torsion function and the first Laplacian eigenfunction
in a convex curvilinear sector in the plane, under homogeneous Neumann
conditions on the two  straight   lateral sides and a homogeneous Dirichlet
condition on the remaining part of the boundary.
We prove
that they are, respectively, strictly $(\frac 1 2)$-concave and strictly
log-concave, provided 
the interior angles  at the vertices are at most $\frac \pi 2$. Under the same assumption,  
  we further establish a billiard-type concavity, obtained through reflections
across the Neumann sides.
As a consequence, we deduce that there exists a unique hot spot, located at
the Neumann--Neumann vertex, and that some monotonicity properties hold
along suitable segments.
Finally, we prove that the associated variational energies, namely the mixed
torsional rigidity and the first mixed Laplacian eigenvalue, satisfy
Brunn--Minkowski type inequalities in the class of convex curvilinear sectors
with a fixed opening angle.

\end{abstract} 

\maketitle

\section{Introduction} 
 The main aim of this paper is to investigate the concavity of solutions to elliptic problems posed in convex curvilinear sectors in the plane under mixed Dirichlet--Neumann boundary conditions, and subsequently to identify the location of their hot spots, together with the geometry of the heat distribution around the maximum.    
  More precisely, we consider bounded convex subsets of 
$\R ^2$ whose boundary consists of two line segments sharing a common endpoint, and a curve. On the two segments we impose homogeneous Neumann boundary conditions, while on the curve we impose a homogeneous Dirichlet condition.

Starting from the 1989 paper \cite{BerPac89} by Berestycki and Pacella on the symmetry properties of solutions  to elliptic boundary value problems of similar   kind,   their qualitative behavior has been extensively studied. Besides symmetry, the focus has been on monotonicity properties \cite{CY18,  YCG21,YCL18,  Zhu01}, and more recently also on related Serrin-type problems \cite{ PPR24, PT20, PT21}.
Particular attention has been attracted, over the past few years, by the case of
the first Laplace eigenfunction on triangles under mixed boundary conditions \cite{AR23, LR17,S16}.
The motivation is the natural link, occurring via its nodal line, with the second Neumann Laplace eigenfunction of triangles,
especially in connection with the hot spots conjecture: formulated by Rauch in 1974, it asserts that its extrema are attained exclusively on the boundary (see \cite{BB99, JM20} and references therein). We point out that a similar hot spot question has been raised 
within the framework of the Polymath research thread \cite{PM}, for the first Laplace eigenfunction of a triangle with one Dirichlet side and two Neumann ones. 
In this case, the question, recently studied in \cite{LY26, Hat24, AR25, CGY26}, 
is whether the maximum occurs only at the vertex opposite to the Dirichlet side.

\smallskip
In this context, to the best of our knowledge, no concavity results are available for PDEs with mixed boundary conditions.  
In fact, the study of concavity properties of solutions to elliptic PDEs has a long history, which is largely devoted to Dirichlet boundary value problems. 
It began in the 1970s, when Makar-Limanov proved that the torsion function is power-concave in planar convex domains
\cite{maklim}, and Brascamp--Lieb established that the first Dirichlet Laplacian eigenfunction is log--concave \cite{BL76}. Since then, several methods, 
now classical, 
have been developed to treat more general PDEs, such as the concavity principle by Korevaar \cite{K83}, the method of continuity by Caffarelli--Friedman \cite{CafFri, KorLew}, and the convex envelope method of Alvarez--Lasry--Lions \cite{ALL} (for further references, see 
the monograph \cite{Kbook} and the more recent survey \cite{GuanMa}). 
All these methods rely on some a priori knowledge of suitable concavity properties near the boundary, which are typically available only under Dirichlet type
conditions. The sensitivity of concavity properties to different boundary conditions and to the regularity of the domain is highlighted by the failure of log--concavity of the first Robin Laplacian eigenfunction on most planar polygons when the Robin parameter is positive and sufficiently small \cite[Theorem 1.2]{AnClHa}. This phenomenon is in fact related to the lack of concavity of the solution to the torsion problem with constant Neumann boundary conditions, which emerges in the asymptotic analysis as the Robin parameter varies near $0$ \cite[Corollary 8.3]{AnClHa}.
On the other hand, the log--concavity of the first Robin eigenfunction can be recovered, in any space dimension, when the convex domain is smooth and the Robin parameter is sufficiently large \cite{CFrobin}. We shall return to Robin boundary conditions at the end of this Introduction, where we propose some natural open questions related to our results. 

\smallskip We are going to confine our attention to the mixed torsion and Laplacian eigenvalue problems in a   convex curvilinear sector in the plane. Namely, we let $\Om$ be a subset of  $\R^2$  
such that $\partial \Om$ contains two (open)  line segments $S'$ and $S''$ with a common endpoint   $O$; then we decompose $\partial \Om$ as 
$$\partial \Om = \Gamma_N \cup \Gamma _ D \cup O\, , \qquad \text{ with } \quad 
\Gamma_N := S' \cup S''\,, \quad \Gamma_D :=  {\partial \Om \setminus (\Gamma_N \cup  O) }   \,.$$
In particular, this framework includes the cases when $\Om$ is a triangle with two Neumann sides and a Dirichlet one, or a circular sector in which the Neumann part of the boundary is the angular one, and the Dirichlet part is the arc of circle.
   Denoting by $\nu$ the outward unit normal to $\Om$,  we consider the following boundary value problems:%
\begin{equation}\label{f:mixedpb} 
\begin{cases}
- \Delta u = 1 & \text{ in } \Om 
\\
u = 0 & \text{ on } \Gamma _D \\ 
\partial _\nu u = 0 & \text{ on }  \Gamma _N  \,,
\end{cases}
\qquad\qquad 
\begin{cases}
- \Delta u = \lambda(\Om) u  & \text{ in } \Om 
\\
u = 0 & \text{ on } \Gamma_D 
\\ 
\partial _\nu u = 0 & \text{ on }  \Gamma_N \,.
\end{cases}
\end{equation}
We refer to the unique solution of the left hand side problem as the {\it mixed torsion function} of $\Omega$. 
As for the  right hand side problem, we assume that  $\lambda(\Omega)$ is the first eigenvalue, 
i.e., the smallest positive value for which a nontrivial solution exists. Then such a solution has a constant sign, and it is 
unique up to a multiplicative factor; when it is nonnegative and normalized in $L ^ 2(\Om)$, we refer to it as the 
 {\it mixed Laplacian ground state} of $\Omega$. 

The junction points  $\Gamma_D \cap  \overline{  S' }  $ and $\Gamma_D \cap   \overline{  S'' }  $, 
will be called the (DN)-vertices  of $\Omega$, while the junction point 
$O :=  \overline{ S' }     \cap     \overline{ S'' }  $ will be called  the (NN)-vertex.  
By {\it interior angle of $\Omega$ at a vertex $p$}, we mean the opening angle of the tangent cone of $\Omega$ at $p$ 
(i.e., the cone generated by $\Omega - p$). 

\smallskip
Our first main result  states that   classical concavity results   for pure Dirichlet problems  continue to hold in the setting of mixed boundary conditions:

\begin{theorem}\label{t:concavity}  
Let $\Om\subset \R ^2$ be     a convex curvilinear sector    such that the interior angles at its three vertices are in $(0, \frac{\pi}{2}]$. 
Then: 

--  if $u$ is the mixed torsion function of $\Om$, the Hessian of   $u ^ {\frac 1 2}$ is  negative definite in $\Om$;
\smallskip

--  if $u$ is the   mixed ground state  of $\Om$, the Hessian of $\log u$ is negative definite in $\Om$. 

\end{theorem} 
 We point out that the threshold imposed on the interior (DN) angles is necessary. Indeed, as stated in Corollary \ref{c:hotspot} below, Theorem \ref{t:concavity} implies in particular that the super-level sets of $u$ are nested convex curvilinear sectors meeting $\Gamma_N$ at  right   angles. Such a geometry is impossible when the interior angle at a (DN) vertex exceeds $\frac{\pi}{2}$.
On the other hand,  on a non isosceles triangle, when the interior angle at the (NN) vertex exceeds $\frac{\pi}{2}$,   the convexity of the level sets  is incompatible with the simultaneous presence of two critical points, occurring for the mixed ground state. This follows from \cite[Corollary 1.6]{LY26}, which asserts that in such a situation the hot spot does not coincide with the (NN) vertex.  For the mixed torsion function,  elementary numerical computations suggest the same behaviour.

\smallskip 
The proof of Theorem \ref{t:concavity} is straightforward in the particular case 
when the interior angle at the (NN)-vertex is of the form $\frac{\pi}{n  }$, with $n \in \N$, $n \geq 2$.  Indeed, in this case it follows 
from the classical results in \cite{maklim, BL76}, upon making the following observation: a finite number of reflections of 
$\Om$ and of the solution to the mixed torsion or eigenvalue problem across the Neumann sides
yield, respectively, a convex domain composed of  $2n$   congruent copies of $\Om$,  
and the solution to the corresponding Dirichlet problem on it; see Figure \ref{fig:1} below. 

 \begin{figure}[ht]
 \center
 \includegraphics[width=4cm]{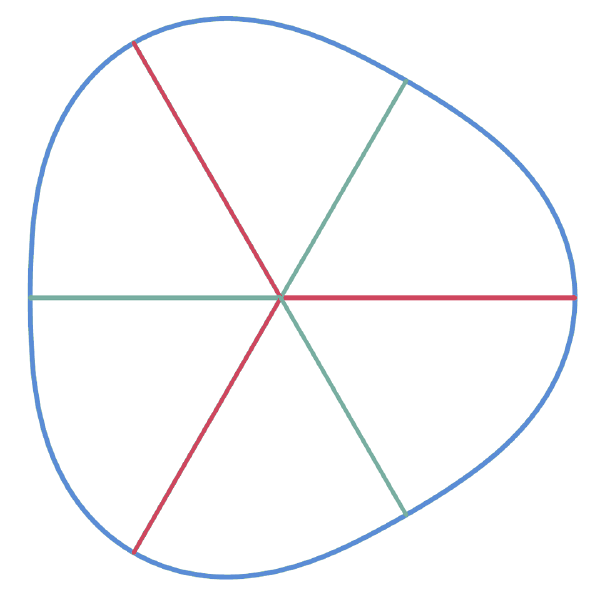}
 \caption{A convex domain made by $6$ copies of  a convex curvilinear sector  with interior  angle  $\frac{\pi}{3}$ at the (NN)-vertex.}     
 \label{fig:1}   
 \end{figure} 
 
However, this argument breaks down as soon as one considers the case of an arbitrary interior angle of $\Om$ at its (NN)-vertex, 
not necessarily of the form $\frac{\pi}{n} $. So, to show that yet the result  holds, we need to provide a different, less elementary proof: it is based on a reworking of the continuity method by Caffarelli--Friedman \cite{CafFri}, 
which involves a local extension of the solution across the Neumann sides 
and an asymptotic expansion near the (NN)-vertex. 
 We point out that this approach can be employed to address the level convexity, monotonicity, and location of hot spots of solutions to more general elliptic PDEs of the form
$\Delta u = f(u,\nabla u)$
in  convex curvilinear sectors   of $\mathbb{R}^2$, when $u$ satisfies the same mixed boundary conditions considered above and suitable structural assumptions are imposed on the source term $f$.

\smallskip 
 
 We now state  some consequences of Theorem \ref{t:concavity}.  
 
 First, we  give  some information concerning the location and uniqueness of the hot spot, as well as the monotonicity of the solutions along certain segments.
 
  Let us mention that,   in the specific case of the first mixed eigenfunction, 
statements (i)-(ii) of Corollary \ref{c:hotspot} are covered by  \cite[Theorem 1.1]{AR25} (dealing more generally with curvilinear triangles);
further results in the same vein were previously obtained in
 \cite[Section 4]{CGY26},  \cite[Theorem 1.1]{Hat24}, \cite[Corollary 1.2 (3)]{LY26}.
 
  Nevertheless, we emphasize that our results go beyond establishing the location and uniqueness of the maximum point: 
they also provide 
a description of the geometry of the surrounding level curves, thereby offering a clear picture of how heat concentrates near the thermal peak.

\begin{corollary}\label{c:hotspot} 
Under the same assumptions of Theorem \ref{t:concavity},  if $u$ is either  the mixed torsion function or the mixed ground state of $\Om$,
and $O$ denotes the (NN)-vertex, it holds:
\begin{itemize}
\item[(i)]  $u$  attains its unique local (and global) maximum over $\overline \Om$ at $O$; 
\item[(ii)]  $u$ is strictly monotone increasing, going from $x$ to $O$,  along any oriented segment  $[x, O]$ contained into $\overline \Om$; 
\item[(iii)] $u$ is strictly monotone increasing, going from $x$ to $y$,  along any oriented segment $[x, y]$ which is contained into $\overline \Om$ and is  orthogonal to $\Gamma _N$ at $y$. 
\end{itemize} 
\end{corollary}

Next, via an unfolding argument, we   establish the concavity behaviour of solutions
along ``billiard trajectories''. To be more precise,   if  $\Om\subset \R ^2$ is  a    convex curvilinear sector  we  say that a function 
$v: \overline{\Omega} \to \mathbb{R}$ is {\it strictly billiard concave} in $\Om$ if the function $v \circ \gamma$ 
is strictly concave on $[0,T]$ whenever $\gamma : [0,T] \to \overline{\Omega}$,  is a {\it billiard trajectory}   parametrized by arc length
  (see Definition \ref{d:traj} for the detailed meaning of the terminology). 
We have:  

\begin{corollary}\label{c:billiardconcavity} Under the same assumptions of Theorem \ref{t:concavity}:

\begin{itemize}
\item[--]  the mixed torsion function $u$ of $\Om$ is  strictly $(\frac 1 2)$-billiard concave (i.e., $u ^ {\frac 1 2}$ is strictly billiard concave);
\item[--]  the mixed ground state $u$ of $\Om$ is strictly log--billiard concave (i.e., $\log u$ is strictly billiard concave). 
\end{itemize} 

\end{corollary}

As a second purpose of the paper,  we  establish  the validity of Brunn--Minkowski type inequalities for the variational energies associated 
with the boundary value problems \eqref{f:mixedpb}, namely the   
 {\it mixed torsional rigidity}, and the {\it first mixed Laplacian eigenvalue}.   They can be  defined respectively by 
 \begin{align} 
& \tau  (\Omega)    = \ds- \!\!\!  \inf _{u \in  H^ {1} _ {\Gamma_D} (\Omega ) }  \int _\Omega\big (    |\nabla u|^ 2 -  2 \, u \big )  \, dx 
   =\sup_{u \in H^ {1} _ {\Gamma_D} (\Omega )  \setminus \{ 0 \}  }
\frac{\left(\displaystyle\int_\Omega u\,dx\right)^2}
{\displaystyle\int_\Omega |\nabla u|^2\,dx} \,, &  \label{f:tauvar} 
 \\ 
& \lambda  (\Omega)    = \ds \!\!\!  \inf _{u \in  H^ {1} _ {\Gamma_D} (\Omega ) }
\frac
{\displaystyle\int_\Omega |\nabla u|^2\,dx} 
 {\displaystyle\int_\Omega | u|^2\,dx} \,, &  \label{f:lambdavar} 
 \end{align} 
 where $H^ {1} _ {\Gamma_D} (\Omega )$ is the subspace of Sobolev functions in
 $H^ {1} _ {\Gamma_D} (\Omega )$ having zero trace on $\Gamma _ D$.

 We prove that the Brunn--Minkowski inequalities for the torsional rigidity and the first Dirichlet Laplacian eigenvalue, due respectively to 
Borell \cite{B85} 
and Brascamp--Lieb \cite{BL76}, extend to the mixed functionals defined in 
\eqref{f:tauvar} and \eqref{f:lambdavar}, 
considered on the class   
$\mathcal K_\alpha$ of  convex curvilinear sectors  with a given interior angle $\alpha \in (0, \frac{\pi}{2}]$ at their (NN)-vertex: 
 
 \begin{theorem}\label{t:BM} Let $\Om _0$, $\Om _ 1$  belong to the class $\mathcal K _\alpha$ for some $\alpha \in (0, \frac \pi 2]$, let  $t \in [0, 1]$, and let $\Om _ t:=(1-t) \Om_0 + t\Om_ 1$. 
 We have:
 \begin{align}
& \tau ^ {\frac 1 4} (\Om_t)   \geq ( 1- t) \tau ^ {\frac 1 4}  ( \Om _0) + t \tau ^ {\frac 1 4}  (\Om_ 1) \, ,    & \label{f:BMt}
\\
& \lambda ^ {- \frac 1 2}   (\Om_t)   \geq ( 1- t) \lambda ^ {- \frac 1 2} ( \Om _0) + t \lambda ^ {- \frac 1 2} (\Om _ 1) \,. & \label{f:BMl} 
\end{align} 
  \end{theorem} 

To our knowledge, in the vast literature on Brunn--Minkowski-type inequalities 
for variational energies (see \cite{C96} and references therein), 
there is no previous contribution involving boundary conditions that are not of pure Dirichlet type.  
To obtain the inequalities \eqref{f:BMt} and \eqref{f:BMl}, 
we use different approaches.  To deal with the case of mixed torsional rigidity, we adapt the
Korevaar--Kennington concavity function method \cite{Kenn,K83}. More
precisely, the inequality \eqref{f:BMt} is obtained as a consequence of the
following comparison result for the mixed torsion functions
$u_{\Om_0}$, $u_{\Om_1}$, and $u_{\Om_t}$, which may be regarded as a
generalized form of the power-concavity statement in
Theorem~\ref{t:concavity}:
$$
u^{1/2}_{\Om_t}((1-t)x+ty)
-(1-t)u^{1/2}_{\Om_0}(x)
-t u^{1/2}_{\Om_1}(y)
\geq 0
\qquad \forall x\in \Om_0,\ y\in \Om_1\, , \   t \in [0, 1] \,. 
$$
During the proof, this inequality is actually proved in a stronger billiard-type version,  by working in a kind of ``unfolded manifold'',
built by glueing together multiple orthogonal reflections of the convex
curvilinear sectors $\Om_0$ and $\Om _1$ across their Neumann sides.

For the mixed eigenvalue, we use a more variational approach,  
inspired by the proof given by Colesanti in \cite[Theorem 1]{C96},   which relies on the construction of a suitable test function via infimal convolution,  to be   employed in the characterization of the eigenvalue 
as the infimum of the Rayleigh quotient.  The key new ingredient in our framework is the identification of the image of the gradient of minus the logarithm of the eigenfunction with the infinite cone generated by the two Neumann sides, see Lemma \ref{l:C1}.   

\smallskip We conclude this Introduction  with the following

\medskip 
 {\it Open questions.}   It would be interesting to understand  to what extent  our results  continue to hold in the following more general settings:
\begin{itemize} 
\item[a.] $\Om$ is a curvilinear triangle in the plane (that is, $\partial \Om$ is formed by the junction of three smooth curves), under homogeneous Neumann conditions on two of them, and homogeneous Dirichlet conditions on the remaining one; 
\item[b.] $\Om$ is a convex curvilinear sector in the plane, under homogeneous Robin boundary conditions on the angular part, and homogeneous Dirichlet conditions on the remaining part;
\item[c.] $\Om$ is a bounded convex cone-like domain in $\R^n$ (that is, a convex body obtained as the intersection of a cone with a compact set), under homogeneous Neumann or Robin conditions on the lateral boundary, and homogeneous Dirichlet conditions on the remaining part. 
\end{itemize}

\section{Proofs of Theorem \ref{t:concavity}, Corollary \ref{c:hotspot}, and Corollary \ref{c:billiardconcavity}}

We start by  a preliminary lemma about the second order local behaviour of the solution to the boundary value problems in \eqref{f:mixedpb} near the (NN)-vertex of $\Om$, when  the interior angle at such vertex is acute. 
 A similar analysis, in case of the mixed ground state and up to the first order, was exploited in \cite[Section 4]{JM20}. 

\begin{lemma}\label{l:expansion}
Let $u$ be either the mixed torsion function or the mixed ground state of a convex curvilinear sector   $\Om$,   whose  
interior angle $\alpha$  at the (NN)-vertex  is strictly smaller than $\frac{\pi}{2}$. 
 Then there exist $r_0>0$ and $c>0$ such that
$u$ is of class  $C^{2,\sigma}\bigl(\Omega\cap B_{r_0}(O)\bigr)$, with 
$\sigma=\frac{\pi}{\alpha}-2>0$, and  satisfies
$$\nabla u(O)=0 \,, \qquad  
\nabla^2 u(x)\leq -c\,I \quad \forall x\in \Omega\cap B_{r_0}(O)\,.
$$
In particular, $u$ is strictly concave in a neighbourhood of $O$.  
\end{lemma} 

\proof 
Let  $\Om$ denote the  curvilinear sector,    and let $\Om _0\subseteq  \Om $ be a circular sector of radius $R$, 
having the same (NN)-vertex as $\Om$, and same interior angle  $\alpha$   at it.   
In  a system of polar coordinates with origin at the (NN)-vertex we have 
$\Om _0 = \{ r \in (0, R) \, , \ \theta \in (0, \alpha) \}$.  Let us consider separately the cases of the mixed torsion and of the first mixed eigenfunction. 

\smallskip
Let $u$ be the mixed torsion function of $\Om$. We let 
$w (r) :=    \frac{r^2}{4}  $, and we consider the function $v := u   +   w$. It satisfies 
$$
\begin{cases}
\Delta v = 0 & \text{ in } \Om _0
\\
v = u  + \frac{R^2}{4}  & \text{ on } \partial \Om _0 \cap \{ r = R \} 
\\ 
\partial _\nu v = 0 & \text{ on }   \partial \Om _0 \cap \{ \theta \in \{ 0, \alpha \}  \}   \,. 
\end{cases} 
$$ 
Thus, by separation of variables, we may write $v$ under the form
$$ 
v ( r, \theta ) = \sum _{ k \geq 0 } a_k r ^ {\mu _ k} \cos (\mu _ k \theta ) \, , 
$$ where, for every $k \in \N$,  
$$
\mu _ k = \frac{ k \pi }{\alpha}  \qquad \text{ and  } \qquad 
a _k  = b _k  R ^ { - \mu _k} \,, 
$$
being $b _k$ the Fourier coefficient of the function $f (\theta):= u ( \theta, R) +  \frac{ R ^ 2 }{4}$. 
Accordingly, 
 \begin{equation}\label{f:expru1} u  
 =   v - w    = u (0) -  \frac{1}{4} r ^ 2 +  \sum _{ k \geq 1 }  a_k r ^ {\mu _ k}   \cos (\mu _ k \theta ) \,.
 \end{equation}  

We observe that the series at the right hand side of \eqref{f:expru1}  
converges uniformly in $[0, \alpha]$  
for every fixed $r \in ( 0, R)$, and
uniformly in any compact subinterval of $[0, R)$ for every fixed $\theta \in (0, \pi)$. 
Indeed, at fixed $r$, it is a Fourier series whose coefficients decrease exponentially 
because,  from the smoothness of $f$, for every $n \in \N$ there exists a positive constant $C _n$ such that 
$|b _ k| \leq C _n \mu _k ^ { -n }$. 
On the other hand, at fixed $\theta$, it is a power series with radius of convergence at least $R$ , since 
$$  \limsup _k  |   a_k   \cos (\mu _ k \theta )  |  ^ {1 /k}    \leq  \limsup _k |   C _n \mu _k ^ { -n } R ^ { - \mu _k}    | ^ {1 /k}   =  R ^ { - \frac{\pi}{\alpha}} < R ^ { -1} .$$ 

Since the same arguments can be repeated for the series obtained by formally differentiating 
with respect to $\theta$ or with respect to $r$,  we infer that 
the gradient and the Hessian of $u$ can be computed differentiating by series. 
To that aim, we use their expressions in polar coordinates, i.e., 
$$ \begin{aligned}
& \nabla u   =  u _ r  e _ r + \frac{1}{r} u  _\theta e _\theta  
\\ 
& \nabla ^ 2 u = u _ {rr} ( e  _ r \otimes e _ r  ) +
\Big ( \frac{1}{r} u _ {r \theta} - \frac{1}{ r ^ 2} u _ \theta  \Big )  ( e _ r \otimes e_ \theta + e _ \theta \otimes e_r) 
+ \Big ( \frac{1}{r^2} u _ {\theta \theta} + \frac{1}{ r } u _ r \Big )   ( e _ \theta \otimes e _\theta)\,. 
\end{aligned}
$$ 
Taking into account that $\mu _ 1 = \frac{\pi}{\alpha} > 2$, we obtain the following asymptotic expansions as $r \to 0 ^ +$
(where the infinitesimal are uniform with respect to $\theta$): 
$$
\begin{aligned} 
& u _ r = \Big [ -\frac{r}{2}
+\sum_{k\ge1} a_k\mu_k\, r^{\mu_k-1}\cos(\mu_k\theta)
 \Big ]   = -\frac{r}{2} + o ( r)   \, , 
 \\
 & \frac{ u _ \theta }{r}  =  \Big [ -\sum_{k\ge1} a_k\mu_k\, r^{\mu_k-1}\sin(\mu_k\theta)
 \Big ]  = O ( r) 
 \\
&  u_{rr}(r,\theta)
=
-\frac{1}{2}
+
\sum_{k\ge1}
a_k\,\mu_k(\mu_k-1)\,
r^{\mu_k-2}\cos(\mu_k\theta) = - \frac{1}{2} + o ( 1) 
\\
& \frac{1}{r}u_{r\theta}
-
\frac{1}{r^2}u_\theta
=
\sum_{k\ge1}
- a_k \mu_k^2 r^{\mu_k-2}\sin(\mu_k\theta)
+
\sum_{k\ge1}
a_k \mu_k r^{\mu_k-2}\sin(\mu_k\theta) = O ( 1) 
\\ 
& 
\frac{1}{r^2} u_{\theta\theta}
+
\frac{1}{r} u_r
=
-\frac{1}{2}
+
\sum_{k\ge1}
a_k \mu_k(1-\mu_k)
\, r^{\mu_k-2}
\cos(\mu_k\theta) =  - \frac{1}{2} + o ( 1) \,.
 \end{aligned} 
 $$ 

We infer that $u$ is of class $C ^ {2, \sigma}$  in a neighbourhood of the origin (with $\sigma = \frac{\pi}{ \alpha} -2$),  it has  vanishing gradient 
   ($\nabla u =  o ( 1)$), and is strictly concave  ($\nabla ^ 2 u = - \frac{1}{2} {\rm Id} + o ( 1)$).

 \medskip
Let now $u$ be the mixed ground state of $\Om$. Denoting by $v$ the restriction of $u$  to the angular sector $\Om _0$ introduced above, we have
$$
\begin{cases}
- \Delta v = \lambda ( \Om) v & \text{ in } \Om _0
\\
v = u   & \text{ on } \partial \Om _0 \cap \{ r = R \} 
\\ 
\partial _\nu v = 0 & \text{ on }   \partial \Om _0 \cap \{ \theta \in \{ 0, \alpha \}  \}   \,. 
\end{cases} 
$$  

Working  in  a system of polar coordinates with origin at the (NN)-vertex, by separation of variables 
we find the same angular part as in the previous case,  while the
radial component must now satisfy  the Bessel equation
$$r ^ 2 R'' ( r) + r R' ( r) + (\lambda (\Om)  r ^ 2 - \mu _k ^ 2 ) R ( r) = 0 \,, \qquad \text{ with } \mu _ k = \frac{ k \pi }{\alpha} \,.$$ 

Thus, in such system of coordinates,  $u$ can be written near the origin under the form
$$ 
u ( r, \theta ) = \sum _{ k \geq 0 } a_k J_{\mu _ k} (\sqrt { \lambda (\Om) }  r)  \cos (\mu _ k \theta ) \, , $$
where $J_{\mu _ k}$ is  the Bessel function of order $\mu _k$.   Recalling  the asymptotic expansion  \cite{Leb72} 
$$J _{ \nu } ( r) = r ^ {\nu} \sum _ { m \geq 0} \frac{(-1) ^ m}{2 ^ { 2m  + \nu } \Gamma ( m + {1 }) \Gamma ( m + {\nu} + 1) }  r ^ { 2m}\,, $$

we see that $J_{\mu _ k} (\sqrt { \lambda  (\Om) }  r)  =  r ^ {\mu _ k}  g _{\mu _ k}   ( r ^2)$, 
where $g _{\mu _k}$ is an entire function with all the Taylor coefficients different from zero. 
 In particular, $g_ 0 ( r^2) = 1 - \frac{\lambda (\Om)}{4}  r ^ 2 +o ( r ^ 2)$. 
Hence, 
 \begin{equation}\label{f:expru2}  
u ( r, \theta ) = \sum _{ k \geq 0 } a_k r ^ {\mu _k} g _{\mu _k} ( r ^2)  \cos (\mu _ k \theta )  = a _0 g_0 ( r ^ 2) +  \sum _{ k \geq 1 } a_k r ^ {\mu _k} g _{\mu _k} ( r ^2)  \cos (\mu _ k \theta ) \, .
\end{equation}

We can now argue in analogous way as done above for the mixed torsion function, starting from  the series expansion \eqref{f:expru2} in place of \eqref{f:expru1}: 
we obtain again that   $u$ is of class $C ^ {2, \sigma}$ in a neighbourhood of the origin (with $\sigma = \frac{\pi}{ \alpha} -2$),  
 it has  vanishing gradient 
   ($\nabla u =  o ( 1)  $) and is strictly concave  ($\nabla ^ 2 u = - a_0 \frac{ \lambda (\Om) }{2} {\rm Id} + o ( 1) )\,.$
\qed 
 
\bigskip \bigskip

{\bf Proof of Theorem \ref{t:concavity}}. 
It is not restrictive to prove the result under the following assumptions:  
the 
interior angles  of $\Om$ are strictly smaller than $\frac{\pi}{2}$  at the (NN)-vertex, and equal to $\frac{\pi}{2}$ 
at the (DN)-vertices, and the Dirichlet portion of the boundary is smooth. Indeed, 
the result when the interior angle at the (NN)-vertex is equal to $\frac{\pi}{2}$ is 
immediately obtained
 by  the  reflection  argument  described in the Introduction.  On the other hand, the  result
 when the angles at the (DN)-vertices  take  values in $(0, \frac{\pi}{2})$  and   the Dirichlet arc  is  nonsmooth, can be obtained by an approximation argument; notice that the preservation of the {\it strict} positivity after passing to the limit, follows from  the constant rank theorem   \cite[Theorem 1]{KorLew}     and Lemma \ref{l:expansion}.   (In particular, the applicability of the constant-rank theorem to our problems will be discussed later during the proof.)

\smallskip 
 Let $\Om _0$ be a circular sector contained into $ \Om$,  having  the same (NN)-vertex $O$ as $\Om$, and 
Neumann sides  lying on the same half-lines containing the Neumann sides of  $\Om$.  
 For $t \in [0, 1]$,  we consider a one-parameter family of   convex curvilinear sectors  $\Om _ t$, which starts from $\Om _0$
 and arrives at   $\Om _ 1 = \Om$, by 
  increasing monotonically under domain inclusion,  and
 varying continuously with respect to the Hausdorff distance. 
 Moreover we construct the family so that, for every $t \in [0, 1]$, 
  the interior angles of $\Om_t$ at its (DN) vertices are equal to $\frac{\pi}{2}$, and the Dirichlet portion of $\partial \Om _ t$ is smooth.

Now, for every $t\in [0, 1]$, we construct a  convex curvilinear sector   $\widetilde \Om_t$ slightly larger than $\Om _t$, 
and an extension to $\widetilde \Om_t$ of the mixed torsion function or mixed ground state of  $\Om _t$. 
To that aim, let us denote by $S'_t$ and $S''_t$ the Neumann sides of $\partial \Om _t$, and by $\Gamma _ t$ its Dirichlet portion. 
 Then we consider  
a    convex curvilinear sector   $U'_t$ of the following kind:
\begin{itemize}
\item[--] 
 $U'_t$  is contained into $\Om_t$; 
\smallskip

 \item[--] $U'_t$ has   the same (NN)  vertex as $\Om_t$; 
\smallskip

 \item[--] 
 $\partial U'_t$ contains both $S'_t$ and 
a small portion of $\Gamma _ t$,  obtained as the intersection between   $\Gamma _ t$ and a neighbourhood of the 
(DN)-vertex $\Gamma _ t \cap   \overline S'_t  $ of $\Om _t$.

\end{itemize} 
Let $\mathcal R'_t (U'_t)$ denote the orthogonal reflection of $U'_t$ across $S'_t$. 
We define $U''_t$ and $\mathcal R '' _t (U''_t)$  in the analogous way, replacing $S'_t$ by $S''_t$. 
Then we set 
\begin{equation}\label{f:junction}  
\widetilde \Om_t:=  \overline \Om_t   \cup \mathcal R' _ t (U'_t) \cup \mathcal R'' _t  (U''_t) \,, 
\end{equation}
see Figure \ref{fig:0}.
We shall denote by $\widetilde S'_t$ and $\widetilde S''_t$ the Neumann sides of $\widetilde \Om _t$, and by $\widetilde \Gamma _ t$ its Dirichlet arc
(which contains by construction the Dirichlet arc $\Gamma _ t$ of $\Om _ t$). 
   We point out  that $\widetilde \Gamma _t $ is smooth: in fact, the orthogonal reflections involved in 
the construction  of $\widetilde \Om _t$ according to \eqref{f:junction} do not create boundary singularities along $\widetilde \Gamma _t$, 
  thanks to the  assumption that the inner angles of $\Om_t$ at its (DN) vertices are equal to $\frac{\pi}{2}$.

\begin{figure} [h] 
     \includegraphics[height=4cm]{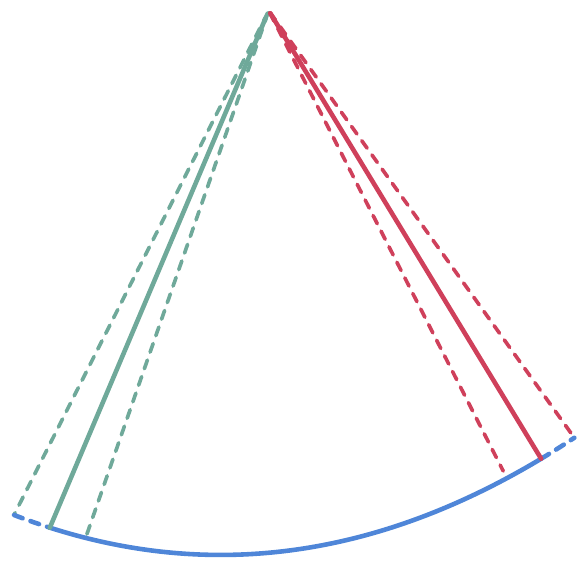}
\caption{The convex  curvilinear   sector  $\Om _ t$  and the enlarged sector $\widetilde \Om _t$ considered in the proof of Theorem \ref{t:concavity}. }
\label{fig:0}   
\end{figure}

 For every $t \in [0, 1]$, we denote by $u _{\Om _t}$ either the mixed torsion function or the mixed ground state of $\Om _t$. 
Then we consider the function $u_t$ defined on the closure of $\widetilde \Om_t$ by
\begin{equation}\label{f:ffunzione} 
u_t:= 
\begin{cases} 
u _{\Om_t} ( ( \mathcal R '  _{t}  ) ^ { -1} ( x))  & \text { if } x \in \overline{ \mathcal R ' _t  (U'_t) } 
\\ \noalign{\smallskip} 
u _{\Om_t}  (x) & \text { if } x \in \overline \Om_t 
\\ \noalign{\smallskip}  
u _{\Om _t} (( \mathcal R '' _ t ) ^ { -1} ( x) )  & \text { if } x \in \overline{ \mathcal R '' _t (U_t '')}\,.$$  
\end{cases}
\end{equation} 
We define on $ \widetilde \Om _t$ the function 
$$v _ t: = \begin{cases}
- u _ t ^ { \frac{1}{2}}  & \text{ when  $u_{\Om_t}$ is the mixed torsion function}
\\  
- \log (u _ t)  & \text{ when  $u_{\Om_t}$ is the mixed ground state}.
 \end{cases} 
 $$
By construction, we have that $v_t$ vanishes on $\widetilde \Gamma _ t$ and  solves
\begin{equation}\label{f:eqvt} 
\Delta v_t  = f ( v _t, \nabla v_t )  \qquad \text{ in }   \widetilde \Om_t\,,
\end{equation} 
being respectively
\begin{equation}\label{f:rhsb} 
f (r, p) = \begin{cases}
- \frac{1}{r} \Big ( \frac{1}{2} + |p| ^ 2  \Big ) & \text{ when  $u_{\Om_t}$ is the mixed torsion function}
\\
\lambda (  \Om_t  ) + |p| ^ 2  & \text{ when  $u_{\Om_t}$ is the mixed ground state}.
\end{cases}  
\end{equation}  
 
Below, we are going to denote  by $H _ t$ the quadratic form  associated with the Hessian matrix $\nabla ^ 2 v _ t$; moreover, 
by writing $H _ t >0$ and $H _ t \geq 0$, we mean respectively that $H _t$ is positive definite and positive semidefinite.  
To prove the result, we are going to show that  
\begin{equation}\label{f:tesihessian}
H _ 1 > 0  \qquad \text { in }  \Om _ 1\,.  
\end{equation} 

To that aim we consider 
 $$\tau :=  \sup \Big \{ t\in   (0, 1 ] \ :\ H _ t >0 \text{ in } \widetilde \Om _ t \Big \}\,,$$ 
 and we prove that necessarily $\tau = 1$.   Once proved that $\tau = 1$,  by using the fact that  $v _ t$  converge to $v _ 1$ 
in $C ^ 2 _{loc} (\widetilde \Om _ 1)$ 
as $t \to  1 ^-$, we obtain that $H _ 1 \geq 0$   in $\Om _1$; then, by invoking the constant rank theorem \cite[Theorem 1]{KorLew} combined with Lemma \ref{l:expansion}, we infer that, in fact, the strict inequality  \eqref{f:tesihessian} holds.  
 We point out that we are in a position to apply the constant rank theorem since,  for any of the two functions $f$ defined in  \eqref{f:rhsb}, we have that   $r \mapsto \frac{1}{f (r , p)}$ is convex: in case of  torsion, such map is affine, while in case of the eigenfunction it is independent of $r$.

 We observe first of all that, from the explicit computation of $u _0$  on a circular sector, it holds
$H _ 0 >0$ in $\widetilde \Om _0$. This ensures that $\tau >0$. 
 Assume by contradiction that $\tau < 1$.  
    Then, from the definition of $\tau$, we have that
  \begin{equation}\label{f:convex} 
  H _ \tau \geq 0 \qquad \text{ in } \widetilde \Om _\tau\,,
  \end{equation} 
  but
  \begin{equation}\label{f:rango1} 
  H _ \tau (x^*) >0 \ \text { is {\it  violated}  for some point }  x^*  \text{ in the closure of } \widetilde \Om _ \tau\,.
  \end{equation} 
 
 We claim that the simultaneous validity of \eqref{f:convex} and \eqref{f:rango1} is not possible. 
 
Let us first examine the case when the point $x ^*$ lies in  $\widetilde \Om _\tau $ (i.e., not on its boundary). 
In this case we observe that,  by \eqref{f:convex}, the function $v _\tau$ is a {\it convex} solution to the PDE in  \eqref{f:eqvt}.    Hence,  we know from 
\cite[Theorem 1]{KorLew} that  the rank of the matrix $\nabla ^ 2 v _\tau$ is constant in $\widetilde \Om _\tau$.  Then, since  $x ^*$  lies in  $\widetilde \Om _\tau $, we infer that  $\nabla ^ 2 v _\tau$ cannot have maximal rank at {\it any} point of $\widetilde \Om _ \tau$. 
But this contradicts the fact that, 
    if $\overline x$  is a point of $\widetilde \Gamma _ \tau $ (different from a (DN) vertex of $\widetilde \Om _\tau$), 
by the same local  argument adopted in the proof of \cite[Lemma 2.4]{K83}
we have 
 \begin{equation}\label{f:rango2}  H _ \tau  (\overline x ) >0 \qquad \text{ in a neighbourhood of } \overline x\,.
 \end{equation} 
 It remains to exclude the possibility that 
the point $x^*$  lies on the boundary of $\widetilde\Om _\tau$. 
Let us show that this cannot occur, by decomposing  such boundary as
$$\partial \widetilde\Om _\tau =  ( \widetilde S_\tau ' \cup   \widetilde S_\tau '' \ )  \cup \widetilde \Gamma _ \tau   \cup O \,.$$     
 \underbar{Case 1.}  $x ^*\in  \widetilde S_\tau ' \cup   \widetilde S_\tau '' $.  
  
  In this case, by reflecting $x^*$ across $S'_\tau$ or across $S ''_\tau$, we find another point $x ^{**}$, lying in the interior of $\widetilde \Om _ \tau $,    such that  the condition $H _ \tau  (x^{**}) >0$ is violated. 
As shown above, this contradicts  the constant rank theorem.

 \smallskip

  \underbar{Case 2.} $x ^*   \in \widetilde \Gamma _ \tau $.
  
If $x^*$ is not a (DN) vertex of $\widetilde \Om _\tau$,  we have a contradiction,  because at any point $\overline x \in \widetilde \Gamma _ \tau $ 
different from its (DN) vertices, \eqref{f:rango2}  holds.  If $x^*$ is a (DN) vertex of $\widetilde \Om _\tau$,   we can reduce ourselves back to the previous situation: indeed, 
by reflecting $x^*$ across $S' _\tau$ or across $S '' _\tau$, we find another point $x ^{**}$, 
which belongs to $\widetilde \Gamma _\tau $ and is not a (DN) vertex,   
 such that  the condition $H _ \tau  (x^{**}) >0$ is violated.

   \smallskip
  \underbar{Case 3.}   $x ^* = O$.  
  
   This cannot occur since we know from Lemma \ref{l:expansion} that  $H _\tau >0$   in a neighbourhood of $O$. 
  \qed
 
\bigskip

\medskip \medskip
{\bf Proof of Corollary \ref{c:hotspot}.}   Just for definiteness, let us deal with the mixed torsion function of $\Om$  (the proof of   the  statement for the logarithm of the mixed ground state  being completely analogous).

Since $O$ is a local maximum by Lemma \ref{l:expansion}, and  the strict concavity of $u ^ {\frac 1 2}$ shown in Theorem \ref{t:concavity} allows at most one local maximum, $O$ is the unique global maximum.

Statement (ii) readily follows from the fact that, by Theorem \ref{t:concavity} and property (i) already proved, the super-level sets of $u$ are nested convex 
 curvilinear sectors   with vertex at $O$.  

 To prove statement (iii), we observe that  the derivative of the function $u ^ {\frac 1 2}$ along the segment $[x, y]$,   oriented from $x$ to $y$,  is strictly decreasing by Theorem \ref{t:concavity} 
and vanishes at $y$ by the Neumann condition satisfied at $y$. Hence it is  positive    along the segment.  
\qed

\bigskip
 We now turn to the proof of Corollary \ref{c:billiardconcavity}. 
To this end, we first make precise the definition of  billiard trajectory and introduce a further definition for later use. 

\medskip
\begin{definition}\label{d:traj} 
A  {\it billiard trajectory} in  a  convex curvilinear sector  $\Om$ with    Neumann boundary $\Gamma _N = S' \cup S''$ and   (NN) vertex at $O$, 
is a 
continuous curve $\gamma: [0, T] \to \overline \Om$  parametrized by arc length  for which there exists a partition
$0 = t_0 < t_1 < \cdots < t_k = T$ such that:

\smallskip
-- for every $ i = 1, \dots, k$, the restriction of $\gamma$ to $(t_{i-1}, t_i)$ is a line segment contained in ${\Omega}$; 

\smallskip
-- for every $ i = 1, \dots, k-1$, the point $\gamma ( t_i)$ belongs to $\Gamma _N   \cup  O   $,
and the
 two consecutive segments meeting at  $\gamma(t_i)$ obey the reflection law, namely:
if   $\gamma ( t_i)  \in \Gamma _N  $,   the two segments 
   form equal angles with $\Gamma _N$; 
    if  $\gamma(t_i)= O$,   then $i= 1$, $k = 2$   and,  denoting by 
     $\alpha$  
 the interior angle at the (NN)-vertex, and by $\theta$ 
  the  angle  between $S'$ and $[O, \gamma(t_{0})]$, it holds that
  the angle formed by  $[O,\gamma(t_2)]$ with the side $S'$, if
$\left\lfloor \frac{\pi}{\alpha} \right\rfloor$
is odd, or  with the side $S''$, if
$\left\lfloor \frac{\pi}{\alpha} \right\rfloor$
is even, is equal to
$\left| \pi - \left(\left\lfloor \frac{\pi}{\alpha} \right\rfloor + 1\right)\alpha
+ \theta \right|$. 
      \end{definition}

\medskip

\begin{definition}\label{d:unfolding}
Let $\Om$ be a  convex curvilinear sector  with (NN) vertex at $O$, and let $v$ be a function defined on $\overline \Om$.  
 We denote by  $\mathcal R_+ (\overline \Om) $ and $\mathcal R_- (\overline \Om)$ 
the reflections of $\overline \Om$  across one of the two half-lines starting at $O$  which support the Neumann sides of 
  $\Om$: by convention, 
the sign $+$ is chosen for the reflection in counterclockwise sense.   

For any positive natural number $n$,  we call  {\it positive/negative $n$-unfolding of $\overline \Om$}  the set composed by $n$ congruent copies of $\overline \Om$ defined by 
$$  \mathcal U _{\pm} ^n  (\overline \Om):=   \bigcup _ {j = 0} ^ {n-1} (  \mathcal R  _\pm ) ^j   (\overline \Om) \,, \qquad \text{ where } (  \mathcal R  _\pm ) ^j := \underbrace {\mathcal R \circ \dots \circ \mathcal R}_{  j   \ \rm{times}}\,. $$ 

 Notice for later use that,  when $j$ runs from $0$ to $n -1$, 
  the interior parts of the sets $(  \mathcal R  _\pm ) ^j   (\overline \Om) $ 
may have a nonempty intersection, and this occurs  if and only if $n \alpha >2\pi$   (see Figure \ref{fig:3}).       
Assuming that $n \alpha \leq 2 \pi$,  we define  on $\mathcal U _{\pm}   ^ n (\overline \Om)$ the {\it positive/negative $n$-unfolding of $v$}  by   
$$  \mathcal U_{\pm}   ^ n (v): =   v ((\mathcal R  _{\pm}^ { j} )  ^ {-1}  ( x))    \qquad    \text{ if } 
x \in   (  \mathcal R  _\pm ) ^j ( \overline \Om )     \,.$$ 
\end{definition}

 \bigskip 
{\bf Proof of Corollary \ref{c:billiardconcavity}.}  Just to fix the ideas we prove the statement  for $v := u ^ {\frac 1 2}$, being $u$ the mixed torsion function of $\Om$. 
We let $n = n _\alpha \in \N$ be the unique integer  such that  
$$
 \alpha\in (\frac \pi{n-1}, \frac{\pi}{n-2}]
$$
(since $\alpha \leq \frac{\pi}{2}$, we have $n \ge 4$), and we  consider the positive/negative $n$-unfoldings  $\mathcal U _{\pm}  ^n  (\overline \Om)$  according to Definition \ref{d:unfolding}.   Notice that,  since $(n-2) \alpha \leq {\pi}$,  we have that $n \alpha \le 2 \pi$, so that   no overlap  occurs  among the interior parts of the sets $(  \mathcal R  _\pm ) ^j   (\overline \Om) $ for $j = 0, \dots , n-1$. 
  On the other hand, thanks to the inequality $(n-1) \alpha > 
 \pi$,  by unfolding 
the family of all billiard trajectories in $\Om$ we find segments $[x,y]$ 
which satisfy one of the following inclusions: 
$$\text{ either } [x,y ]\sq   {\rm int} (  \mathcal U _{+} ^  n  (\overline \Om) )  \cup O, \qquad \text{ or } [x,y ]\sq  
  {\rm int}  ( 
\mathcal U _{- } ^ n  (\overline \Om))    \cup O    \,.$$ 

Thanks to this unfolding procedure,  in order to obtain the strict billiard-concavity of $v$ in $\Om$,  it is enough to show that the Hessian matrices of 
the functions 
$\mathcal R_{\pm, n} (v) $  are positive definite  in the interior parts of    $  \mathcal U _{\pm}  ^ n  (\overline \Om)$   and at $O$.   
To that aim, it is enough to observe that  such Hessian matrices are positive definite in   $(  \mathcal R  _\pm ) ^j   (\overline \Om)$ 
for every $j= 0, \dots, n-1$ (this follows from 
Theorem \ref{t:concavity} applied to any of  the convex   curvilinear sectors  
$(  \mathcal R  _\pm ) ^j   (\overline \Om)$), and  that the unfolded functions 
 $\mathcal U_{\pm} ^ n (v) $  
 are of class $C ^ 2$ 
across the Neumann boundary of $(  \mathcal R  _\pm ) ^j  (\overline \Om)$  and at $O$ 
(respectively
by the  homogeneous Neumann condition satisfied by $v$ on $\Gamma _N$ and by Lemma \ref{l:expansion}).  \EEE 
 \qed

\section {Proof of Theorem \ref{t:BM}} \label{sec:reflections}

\subsection{Proof of Brunn--Minkowski inequality \eqref{f:BMt} for the mixed torsional rigidity}

Similarly as done in the proof of Theorem \ref{t:concavity}, 
it is not restrictive to prove the result under the assumptions  that ${\pi}/{\alpha}$ is not an integer,   
that $\Om _0$, $\Om _1$ have interior angles  at their (DN)-vertices equal to ${\pi}/{2}$, and that the Dirichlet portion of their boundary is  
of class $C ^ 2$ and has strictly positive curvature.  
  Notice that, since these additional conditions are closed under Minkowski addition, they will be satisfied also by $\Om _t$.  

\smallskip 
The inequality \eqref{f:BMt} is trivially true if $t = 0$ or $t = 1$. So we let $t \in (0, 1)$, which will remain fixed throughout the proof. 
 We are going to show that,  denoting by 
 $u _{\Om_t}$  the mixed torsion function of $\Om _ t$, it holds
 \begin{equation}\label{f:Kenn} 
u _{\Om_t} ^ { 1 /2} (  ( 1 -t ) x + t y)  - ( 1- t)  u _{\Om_0} ^ { 1 /2} ( x)  - t u _{\Om_1} ^ { 1/ 2} ( y ) \geq 0 
\qquad \forall x \in \Om _0, \ y \in \Om _1\,.
\end{equation} 

Once proved this inequality, the result readily follows by arguing as in the proof of \cite[Theorem 11]{C96}. 
Namely,  after extending to zero $u _{\Om_t} $ outside $\Om _t$,  the arithmetic-geometric inequality gives
$$u _{\Om_t} (  ( 1 -t ) x + t y)  \geq   u _{\Om_0} ( x)^ { 1- t}    u _{\Om_1} ( y ) ^ t  \qquad \forall x ,  y \in \R^2 \,.$$
Then  \eqref{f:BMt}  follows by applying  Pr\'ekopa-Leindler inequality \cite{PL1, PL2}. 

We observe that inequality \eqref{f:Kenn} can be rephrased as
\begin{equation}\label{f:goall} \C  (x, y)\geq 0 \qquad \forall x \in \Om _0, \ y \in \Om _1\,,  \end{equation} 
where
  $\C $ denotes  
 the generalized Korevaar-type concavity function:
   \begin{equation} \label{f:gencol} 
   \C  (x, y):=   u ^ { 1/2} _{\Om_t} (  ( 1 -t ) x + t y)  - ( 1- t)   u  ^ {1/2} _{\Om_0} ( x)  - t  u  ^ { 1/2} _{\Om_1}  ( y) \qquad \forall x \in \Om _0, \ y \in \Om _1\,.
   \end{equation}

 In order to prove \eqref{f:goall}, we are going to solve a minimization problem for a suitable ``unfolding'' of the concavity function $\C$. 
The formulation of such problem requires some preliminaries. 
Just to simplify the notation, in the remaining of the proof we set 
$$K_0:= \overline \Om _0\, , \qquad K _1:= \overline \Om _1, \qquad K _t:=  \overline \Om _t\,, $$
and 
$$v_0:= u ^ { 1/2}_{\Om _0}\, , \qquad v_1:= u ^ { 1/2}_{\Om _1}\, , \qquad v_t:= u ^ { 1/2}_{\Om _t} \,.$$ 
We choose 
 $n= n_ \alpha \in \N$  as in the proof of Corollary \ref{c:billiardconcavity},  that is as the unique integer such that 
\begin{equation}\label {f:choicen}
  \alpha\in \big (\frac \pi{n-1}, \frac{\pi}{n-2} \big ]\,.
\end{equation} 
Recall that, since $\alpha \in (0,  \frac{ \pi}{2})$, we have $n \geq 4$.  

For such $n$, we are going to consider the 
$(n+1)$-positive unfoldings   of the  sets $K_0$, $K_1$ and $K_t$, and of the functions $v_0, v_1, v_t$ defined on them.
  The reason why we choose  $(n+1)$  as the number of reflections will be explained later during the proof.  
Since we are going to work only with {\it positive} unfoldings,  when we adopt the notation introduced in Definition \ref{d:unfolding}, for simplicity 
we shall omit the symbol $+$. 
For reasons which will be clear below, we deal separately with the cases $n \geq 5$ and $n = 4$.

\medskip 
$\bullet$ 
Let $n \geq 5$, which corresponds to an angle $\alpha \in (0,  \frac{\pi}{3}]$.  Since the upper inequality in  \eqref{f:choicen}   gives $ ( n + 1 ) \alpha \leq   \frac{n+1}{n-2}  \pi$ ,  we have 
$(n +1 ) \alpha \leq 2 \pi$.  Hence, in the definition of the $(n+1)$-unfolding of $K _0$ given according to Definition \ref{d:unfolding}, namely  
\begin{equation}\label{f:disjoint} 
\mathcal U ^{n + 1} ( K _0):= \bigcup _ {j = 0} ^ {n} (  \mathcal R ) ^j   (K_0) \,, \qquad \text{ where } (  \mathcal R ) ^j:= \underbrace {\mathcal R \circ \dots \circ \mathcal R}_{j \ \rm{times}} \,,
\end{equation} 
the sets  $\mathcal R ^ j  ( K _0)$, for $j = 1, \dots, n$, 
   have pairwise disjoint interiors. 
   
 Notice also that, since the lower   bound for $\alpha$    in \eqref{f:choicen}
ensures that $(n+1) \alpha >  \pi$,  the set $\mathcal U ^{n + 1} ( K _0)$ turns out to be non convex with an interior concave corner at $O$
  (except for the limit case $n = 5$, $\alpha = \frac{\pi}{3}$), see the example in Figure  \ref{fig:2}.    
 \begin{figure}[ht]
 \center
 \includegraphics[width=4.3cm]{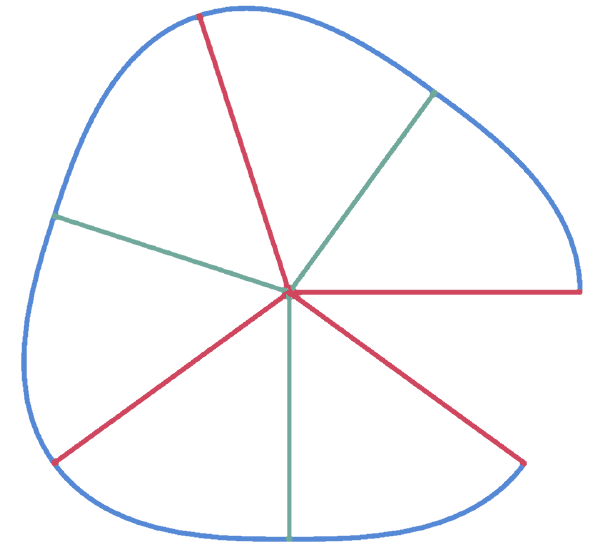}
 \caption{The $6$-unfolding of a convex curvilinear sector with  interior (NN)--angle 
  $\alpha = \frac{3\pi}{10} \in (\frac{\pi}{4}, \frac {\pi}{3})$ }  
 \label{fig:2}   
 \end{figure}

Thanks to the fact that the sets  $\mathcal R ^ j  ( K _0)$, for $j =0, 1, \dots, n$, 
   have pairwise disjoint interiors,  we can proceed by defining the positive $(n+1)$-unfolding of $v_{0}$  on $\mathcal U ^{n + 1} ( K _0)$   according to 
 Definition \ref{d:unfolding}, namely  
\begin{equation}\label{f:disjoint1} 
\mathcal U ^ {n+1} (v_0): =  v _0  ((\mathcal R  ^ { j} )  ^ {-1}  ( x))    \qquad    \text{ if } 
x \in     \mathcal R   ^j ( K_0 )     \,.
\end{equation}   
  
  In the same way, we define  $\mathcal U ^ {n+1} (v_1) $ on $\mathcal U^ {n+1} (K_1)$ and  $\mathcal U^{n+1} (v_t) $  on $\mathcal U ^ {n+1} (K_t)$. 
  
  \smallskip 
  Since we shall have to compute $\mathcal U^{n+1} (v_t) $   on Minkowski combinations in proportion $t$ of points in  
  $\mathcal U ^{n+1} (K_0)$ and $\mathcal U ^ {n+1} (K_1)$, we need to establish some sufficient conditions in order that these combinations belong to  
  $\mathcal U ^ {n+1} (K_t)$.  To that aim, we prove Claim I. below. Notice carefully that    our convex curvilinear sectors $K_0, K _1 \in \mathcal K _\alpha$  
satisfy  the assumption 
 \eqref{f:ipalfa} below,  by taking therein
$\omega=\alpha$ and  $h$ equal to the integer $n = n _\alpha$ chosen as in \eqref{f:choicen},  only provided 
 $\alpha \in (0, \frac{\pi}{3}]$,  which corresponds to   
 $n \geq 5$.  Indeed, for $n$ odd   we have 
$  \frac{\pi}{n-2} \leq   \frac {2 \pi } {n +1}$ provided   $n \geq 5$, while 
for $n$ even   we have $  \frac{\pi}{n-2} \leq   \frac {2 \pi } {n +2}$  provided  $n \geq 6$.

\medskip 

\underbar{Claim I.}: {\it   Let 
$K_0, K_1$ be convex curvilinear sectors with interior angles at their (DN) vertices equal to $\frac{\pi}{2}$, and interior angle $\omega >0$ at their  (NN) vertex.     
For every $h \in \N$, $h \geq 1$,  such that  
 \begin{equation}\label{f:ipalfa} 
(h+1)   \omega  \in (0,  2\pi] \text{ if $h$ is odd, }  \qquad   (h+2)  \omega    \in (0, 2\pi] \text{  if $h$ is even, }
\end{equation} 
letting 
$$\begin{aligned} 
\mathcal A^{h+1} (K_0, K_1):= \Big \{
& (x, y) \in \mathcal U ^ {h + 1} ( K _0) \times \mathcal U ^{h + 1} ( K _1) \ :\ 
\text{either }  xy = 0 , \text{ or }
\\  
 & x=\rho_x e^{i \theta_x} \, , y=\rho_y e^{i \theta_y},  \text{  for  } \theta _x, \theta_y \in \R  \text{ with }  
 \theta_x \vee \theta_y- \theta_x \wedge \theta_y\le \pi \Big \} \,,
 \end{aligned}
 $$  
it holds  that 
\begin{equation} \label{f:claim} 
(x, y) \in \mathcal A^{h+1} (K_0, K_1)\, , \ t \in (0, 1) \ \Rightarrow \ 
( 1-t) x + t y\in \mathcal U ^{h + 1} ( K _t)\,. 
\end{equation} }

\underbar{Proof of Claim I.} We argue by induction on $h$, starting by $h=1$ up to its maximal admissible value. Let us  start by proving the claim for  $h= 1$, with $  \omega  \in (0, \pi]$. 
For 
$  \omega =\pi$  the conclusion is immediate: indeed
in this case the sets $\mathcal U ^ 2 ( K _0)$, $\mathcal U ^ 2 ( K _1)$, $\mathcal U ^ 2 ( K _t)$  are convex, 
and we are reduced back to a classical Minkowski combination since 
$$ \begin{aligned} \mathcal U ^ 2 ( K _ t) & =  [(1-t)K_0+t K_1]\cup  \mathcal R [(1-t) K_0+t K_1] \\ & = 
(1-t)[K _0\cup  \mathcal R (K_0)] + t[K_1\cup  \mathcal  R (K_1)]  
 = 
( 1-t) \mathcal U ^ 2 ( K _0) + t  \mathcal U ^ 2 ( K _ 1). 
\end{aligned} 
$$  
Let $ \omega \in (0, \pi)$.  If  $(x,y) \in K _0 \times K _ 1$ or $(x,y) \in \mathcal R( K _0 ) \times \mathcal R( K _ 1)$, 
\eqref{f:claim} holds because we have respectively  $z _t \in K _t $ or $z _t \in \mathcal R ( K _ t)$. 
Thus, up to changing labels,  we can assume without loss of generality that
$$x \in K _0 \qquad  \text{ and } \qquad  y \in \mathcal R ( K _1)\, , $$

\smallskip
We observe firstly that, if  $X _{ \omega   }$ is the infinite cone generated by the half-lines supporting the Neumann sides of $K _0$, $K _1$, 
denoting by  $\mathcal U ^ {2} (X _{  \omega   })= X_{  \omega  } \cup \mathcal R (X_{ \omega   }) $ its $2$-positive unfolding,   since $(x, y) \in \mathcal A ^ 2 ( K_0, K _ 1)$, 
there holds
$$z_ t:= (1-t)x +t y  \in \mathcal U ^ 2 ( X _{ \omega })\,.$$ 

Let $x':= \mathcal R(x)\in \mathcal R (K_0)$ and $y':= (\mathcal R)^{-1}(y)\in K_0$. Then, 
$$\begin{aligned}
&z_t ' := (1-t)x '+t y \in \mathcal R (K_t) \subset \mathcal U ^ 2 ( K _ t) 
\\
& z_t'':= (1-t)x +t y' \in K_t  \subset \mathcal U ^ 2 ( K _ t)  \,,
\end{aligned} 
$$ 
 Since $z_t \in [z_t',z_t'']$, the conclusion that
$ z_t \in \mathcal U^2(K_t)$   (see Figure \ref{fig:4}, left)   follows from the following observation,
applied with   $K$ equal to $K _t$ or  $\mathcal R ( K _t)$, respectively in case $z_t$ belongs  to $X _\omega$ or  $\mathcal R ( X_\omega)$. 

 Let $K$ be the closure of a convex curvilinear sector with interior
angle $\omega\in(0,\pi)$ at its (NN) vertex and interior angles not
exceeding $\pi/2$ at its (DN) vertices. Let
${\rm cone}(S',S'')$ be the infinite cone generated by the Neumann
sides $S'$ and $S''$ of $K$. Then, for every $z\in K$, if
$\overline z$ denotes the orthogonal projection of $z$ onto either
of the lines supporting $S'$ or $S''$, 
due  to the convexity of $K$ and the assumption that
the (DN) angles are not larger than $\pi/2$, it holds that $
 [z,\overline z]\cap {\rm cone}(S',S'') \subset K$    (see Figure \ref{fig:4}, right). 
  
   \begin{figure}[ht]
 \center
 \includegraphics[width=10cm]{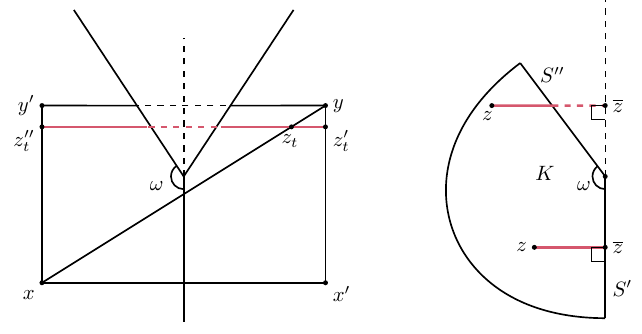}
  \caption{ The inclusions  $ z_t \in \mathcal U^2(K_t)$  and  $[z,\overline z]\cap {\rm cone}(S',S'') \subset K$ in the proof of Claim I.}  
 \label{fig:4}   
 \end{figure}

  \smallskip
 Assume now that  the claim is true up to $h -1$, and let us prove it is true for $h$. 
 By inductive assumption, up to changing labels, the unique case we have to deal with is when
\begin{equation}\label{f:unique}  
x \in K_0 \qquad \text{  and  } \qquad y \in \mathcal R^h( K _ 1)\,.
\end{equation} 

Assume first that $h$ is odd. 
In this case,  
we are  reduced back to the two cells case $  h    = 1$ already discussed, simply by replacing $K _0$ and $K _1$ by their
$(\frac{h+1}{2})$-unfoldings
 $ \mathcal U ^{\frac{h+1}{2} } ( K _0)$ and $\mathcal U ^{\frac{h+1}{2} } ( K _1)$.   
  Notice that we are in a position to apply the claim in the two cells cases, 
  because of the assumption that the first condition in \eqref{f:ipalfa} holds.    
 
 If $h$ is even, we can formally add a new reflection, and work with the $(  \frac{h+2}{2}  )$-unfoldings $\mathcal U ^{\frac{h+2}{2}} ( K _0)$ and  
 $\mathcal U ^{\frac{h+2}{2} } ( K _1)$, so that we can again reduce the problem back to the two cell case. 
This can be done thanks to the assumption that the second  condition in \eqref{f:ipalfa} holds,
and the proof of  Claim I. is achieved.

\medskip 
$\bullet$ 
Let now $n  =4$, which corresponds to an angle $\alpha\in ( \frac{\pi}{3}, \frac{\pi}{2}]$.   Then:  

\smallskip

--  If $\alpha   \in ( \frac{\pi}{3} , \frac{2 \pi}{5}]$, the situation is similar to the case $n \geq 5$, i.e.,   the sets $\mathcal R ^ j  ( K _0)$ have pairwise disjoint interiors  for 
$j  = 0, ... ,  4$, so that $\mathcal U ^5 ( K_0)$  and $\mathcal U ^ 5 ( v_0)$  could be still  defined as in \eqref{f:disjoint} and \eqref{f:disjoint1} respectively. 
Again, the set  $\mathcal U ^5 ( K_0)$ is non convex,  except for the limit case  $\alpha =  \frac{2 \pi}{5}$. 

\smallskip
-- If  $\alpha   \in ( \frac{2 \pi}{5}, \frac{\pi}{2}]$, the sets $\mathcal R ^ j  ( K _0)$ still have pairwise disjoint interiors  for $j  = 0, ...,  3$ (since $4 \alpha \leq 2 \pi$), but the interiors of $K_0$ and $\mathcal R ^ 4 ( K_0)$ overlap,  
 see  the example in Figure  \ref{fig:3},   so that    more precisions   are needed to properly define  $v_0$ on  $\mathcal U ^5 ( K_0)$.   

  \begin{figure}[ht]
 \center
 \includegraphics[width=4.3cm]{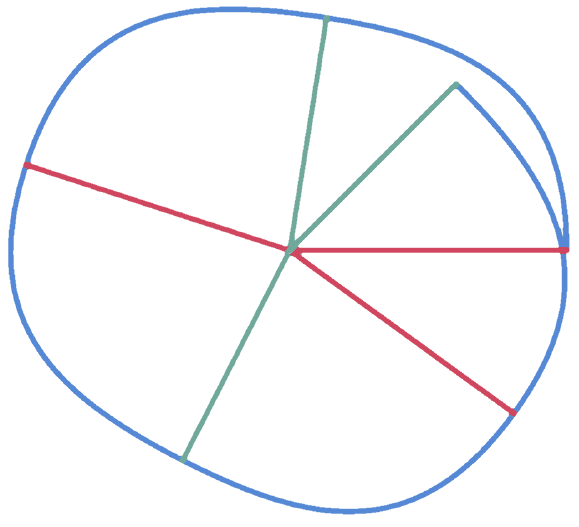}
  \caption{The $5$-unfolding of a convex curvilinear sector with interior (NN)--angle 
  $\alpha = \frac{9\pi}{20} \in (\frac{\pi}{3}, \frac {\pi}{2})$ }  
 \label{fig:3}   
 \end{figure}

  Therefore,  we are led to adopt an ad hoc representation of the
unfolding $\mathcal U ^5 ( K_0)$:  instead of viewing it as an ordinary subset of $\R^2$, 
we  provide a  cut-open angular representation of it, whose purpose 
is to bring us back to the situation in which the   sets
   $ \mathcal R ^j(K_0 )$ 
have pairwise disjoint interior for $j  = 0, ...,   4$. 
 We proceed in this way for every angle $\alpha$ corresponding to $n=4$, namely
$\alpha\in \left(\frac{\pi}{3},\frac{\pi}{2}\right]$, although, as highlighted above,
the procedure is actually needed only for
$\alpha\in \left(\frac{2\pi}{5},\frac{\pi}{2}\right]$.  
   
  In the    complex  plane with origin at $O$, through the identification of the point $\rho e ^ { i \theta}$ with the pair $(\theta, \rho)$,  
 we view $K _0$, for some function $\rho _{K _0}:[0, \alpha] \to \R ^+$,  as 
  $$K_0= \Big \{  (\theta, \rho) \ : \ \theta \in [0, \alpha],\ 0 \leq  \rho \leq    \rho_{K _0}(\theta)] \Big \}\,. $$

By abuse of notation, we denote by $\mathcal R ^ j ( \theta, \rho)$  the $j$-th reflection of the pair $(\theta, \rho)$, meant as
\begin{equation}\label{f:jmap} 
\mathcal R ^ j ( \theta, \rho):= (\theta _j, \rho)\,, \qquad \text{ with }  \theta _j:= \begin{cases}  
(\theta-j \alpha)  & \mbox{ if $j$ is even}
 \\ 
((j+1)\alpha-\theta)  & \mbox{ if $j$ is odd}.
 \end{cases} 
\end{equation} 
  
  Consequently, for every  $j \ge 0$, we view  $\mathcal R  ^j(K_0 )$ as  
 \begin{equation}\label{f:ei} 
 \mathcal R ^j(K_0 )=\Big  \{(\theta , \rho) \  : \  \theta \in [j \alpha , (j+1) \alpha] ,\  0\leq  \rho  \leq  \rho_{K_0}(\theta _j) \Big \}\,.
 \end{equation} 
Once we intend the reflection operator $\mathcal R ^ j$ as in \eqref{f:jmap} and the sets  $ \mathcal R ^j(K_0 ) $ as in \eqref{f:ei}, 
we  can define again $\mathcal U ^ 5 (K_0)$ as in \eqref{f:disjoint} and $\mathcal U ^ 5 (v_0)$  as in \eqref{f:disjoint1};  the latter definition  is now  well-posed, because the sets $ \mathcal R ^j(K_0 ) $, meant as in \eqref{f:ei}, have pairwise disjoint interiors for $n = 0, \dots, 4$, for any $\alpha\in ( \frac{\pi}{3}, \frac{\pi}{2}]$. 
 
\smallskip 

  In the same way, we define  $\mathcal U^{5} (v_1) $ on $\mathcal U^{5} (K_1)$ and  $\mathcal U^{5} (v_t) $  on $\mathcal U^{5} (K_t)$. 

\smallskip 

  Now similarly as done for $n \geq 4$, we provide an analogue sufficient condition in order that the Minkowski combination  in proportion $t$ of points in 
   $\mathcal U^{5} (K_0)$ and $\mathcal U^{5} (K_1)$ lies in $\mathcal U^{5} (K_t)$.   We prove:  
 
  \smallskip 
\underbar{Claim II.}: {\it  For $\alpha \in ( \frac{\pi}{3}, \frac{\pi}{2}]$, and $K_0, K_1 \in \mathcal K _\alpha$ as above,     
letting 
$$\begin{aligned} 
\mathcal A^{5} (K_0, K_1):= \Big \{
& \big ( (\theta_x, \rho_x) ,  (\theta_ y, \rho _y) \big ) \in \mathcal U ^{5} ( K _0) \times \mathcal U ^{5} ( K _1) \ :\ 
\text{either }  \rho_ x \rho _y = 0 , 
\\  
 &   
 \text{ or } \theta_x \vee \theta_y- \theta_x \wedge \theta_y\le \pi \Big \} \,,
 \end{aligned}
 $$  
it holds  that 
\begin{equation} \label{f:claim2} 
\big ( (\theta_x, \rho_x) ,  (\theta _y, \rho _y)\big )  \in \mathcal A^5 (K_0, K_1)\, , \ t \in (0, 1) \ \Rightarrow \ 
\big (\theta _ {( 1-t) x + t y}, \rho _ {( 1-t) x + t y} \big ) \in \mathcal U ^{5} ( K _t)\,. 
\end{equation} }

 \smallskip 
\underbar{Proof of Claim II.}: 
Up to changing labels, the unique case we have to discuss is when
\begin{equation}\label{f:unique1}  
(\theta _x, \rho _x)  \in K_0 \qquad \text{  and  } \qquad (\theta _y, \rho _y)   \in \mathcal R^4( K _ 1)\,.
\end{equation} 
Indeed, in all the other cases, we are reduced to
deal with a family  of  just $4$ reflections of $K_0$ and $K _1$. 
Since $4 \alpha \leq 2 \pi$, in such a situation
we can see elements of 
 $\mathcal R ^ j ( K_0)$ and $\mathcal R ^ j ( K_1)$  back   as points of the form $\rho e ^ { i \theta}$, 
rather than pairs $(\theta, \rho)$.  Then we can 
apply Claim I. for $h = 3$ in order to     deduce that the Minkowski sum $(1-t ) x + t y$ belongs to 
$\bigcup _{j=0} ^ 3  \mathcal R ^ j ( K _t)$, and hence   that \eqref{f:claim2} holds. 

Now, assume that \eqref{f:unique1} holds.
 From the assumptions $(\theta _x, \rho _x)  \in K_0  $ and $\big ( (\theta_x, \rho_x) ,  (\theta _y, \rho _y)\big )  \in \mathcal A^5 (K_0, K_1)$, and since $3 \alpha > \pi$, we have that  necessarily 
$(\theta _y, \rho _y) \in \cup _{j=0} ^ 3  \mathcal R ^ j ( K _1)$.   
Then we are again in a position to apply Claim I. for $h = 3$, so that \eqref{f:claim2} holds, and the proof of Claim II. is achieved.

\bigskip
Thanks to the above analysis,  for every angle $\alpha \in (0, \frac{\pi}{2}]$, choosing $n = n _ \alpha$ as in \eqref{f:choicen}, 
the following definition of  $(n+1)$-unfolding 
of the concavity function \eqref{f:gencol} is well-posed 
for any pair  $(x, y) \in \mathcal A^{n+1} (K_0, K_1)$:
  $$\widetilde \C  (x, y):=    \mathcal U^{n+1}(v_t)  (  ( 1 -t ) x + t y)  - ( 1- t)   \mathcal U^{n+1} (v _0) ( x)  - t \,  \mathcal U^{n+1} (v_1 ) ( y)  \,.$$      
Thus we can consider the minimization problem 
$$m :=\min \Big  \{ \widetilde \C (x, y) \ :\ (x, y) \in  \mathcal A^{n+1} (K_0, K_1) \Big \} \,.$$  
    We are going to show that
   $$m = 0\,,$$
     which in particular implies the inequality \eqref{f:goall}, and may be viewed as a stronger, billiard-type version of it.

     \smallskip
  Since we are minimizing a continuous function over a compact set, $m$ is attained at some pair $(\overline x,  \overline y) \in  \mathcal A^{n+1} (K_0, K_1)$. 
In the remaining of the proof, we set for brevity:
$$\mathcal A:= \mathcal A^{n+1} (K_0, K_1)\, , \qquad \mathcal U _0:= \mathcal U^{n+1} (K_0 ), \qquad \mathcal U _1:= \mathcal U^{n+1} (K_1 ),
\qquad \mathcal U _t:= \mathcal U^{n+1} (K_t )\,.$$   
   
\medskip
\underbar {Case 1}: $(\overline x,  \overline y) \in   {\rm int} ( \mathcal A ) = \mathcal A  \setminus \partial \mathcal A $.

\smallskip 
 Since by Claim I. (for $n \geq 5$) and  Claim II. (for $n = 4$) we know that the  linear map $F (x, y) = ( 1-t ) x + t y$ sends $\mathcal A$ into $\mathcal U _t$, 
we also have that  it sends the interior of $\mathcal A$ into  the interior of $\mathcal U _t$. It follows that, if we perturb the point 
$(\overline x,  \overline y) $, the small perturbation undertaken by the Minkowski combination  $(1-t ) \overline x + t \overline y$ leaves it in the interior of $\mathcal U _t$. 
Then,   we   can reproduce the same perturbation argument used in the proof of Theorem 20 in \cite{C96}, in the case when both points are in the interior part of the corresponding convex bodies. 
We limit ourselves to remark that such argument applies because it is purely local, 
and we refer to that paper for the details: the idea is that the inequality $\widetilde \C (  \overline x,   \overline y) \geq 0$ 
is obtained by using the first and second order optimality conditions coming from the minimality at 
$(\overline  x,   \overline y)$ (i.e., the vanishing of the gradient  of $\widetilde \C$ and the semi-positiveness of its Hessian matrix at 
$(  \overline x,  \overline y)$), and the differential equations satisfied  by the unfoldings of $v_0$ and $v_1$.

\medskip
\underbar {Case 2}:  $(\overline x,  \overline y) \in    \partial \mathcal A $.  
 
\smallskip
We observe that $  \partial   \mathcal A$ can be decomposed as 
$$
\begin{aligned}
\partial \mathcal A=  {} 
& \Big\{
(x,y)\in \mathcal A:
x\in\partial \mathcal U_0\setminus O,\
y\in\partial \mathcal U_1\setminus O
\Big\}
\\
&\cup \Big\{
(x,y)\in \mathcal A:
x\in\partial \mathcal U_0\setminus O,\
y\in{\rm int}\, \mathcal U_1\setminus O 
\Big\}
\\
&\cup \Big\{
(x,y)\in \mathcal A:
x\in{\rm int}\, \mathcal U_0\setminus O,\
y\in\partial \mathcal U_1\setminus O 
\Big\}
\\
& \cup (O,O)  \cup 
\big(O\times (\mathcal U_1\setminus O )\big)
\cup
\big((\mathcal U_0\setminus O )\times O \big)
\\
&\cup \Big\{
(x,y)\in{\rm int}\, \mathcal U_0\times{\rm int}\, \mathcal U_1:
x,y\neq O,\ |\theta_x-\theta_y|=\pi \Big\}.
\end{aligned}
$$ 
In turn, $\partial \mathcal U _0$ and $\partial \mathcal U _1$ can be decomposed as 
$$ \partial \mathcal U _0  = \Gamma _ D (\mathcal U _0) \cup \Gamma _N (\mathcal U _0) \cup  O    \qquad \text{ and }  \qquad 
 \partial \mathcal U _1  = \Gamma _ D (\mathcal U _1) \cup \Gamma _N (\mathcal U _1) \cup  O  \,.$$ 
  
Actually, in our analysis below,  we can ignore the  Neumann portion $\Gamma _N (\mathcal U _0)$, $\Gamma _N (\mathcal U _1)$, by considering their points as
if they were in the interior of $\mathcal U _0$, $\mathcal U _1$. 
Indeed, if $m$ is attained at a pair $(\overline x,\overline y)$ such that
$\overline x \in \Gamma_N(\mathcal U_0)\cap\{\theta=0\}$ or 
$\overline y \in \Gamma_N(\mathcal U_1)\cap\{\theta=0\}$,
then it is also attained at the point
$(\mathcal R^2\overline x,\mathcal R^2\overline y)$,
namely at the counterclockwise rotation by angle $2\alpha$ of 
$(\overline x,\overline y)$. Such rotated point belongs to
${\rm int}(\mathcal A)$, since, by \eqref{f:choicen}, we have
$\pi+2\alpha < (n+1)\alpha$.

The same argument applies if
$\overline x \in \Gamma_N(\mathcal U_0)\cap\{\theta=(n+1)\alpha\}$ or 
$\overline y \in \Gamma_N(\mathcal U_1)\cap\{\theta=(n+1)\alpha\}$.
In this case, one replaces $(\overline x,\overline y)$ by
$(\mathcal R_-^2\overline x,\mathcal R_-^2\overline y)$,
namely by the clockwise rotation by angle $2\alpha$ of the pair
$(\overline x,\overline y)$.
This possibility of reducing points on the Neumann boundary to interior
points by means of a double reflection is precisely the reason why, in this
proof, we work with $(n+1)$-unfoldings rather than with the
$n$-unfoldings used in Corollary~\ref{c:billiardconcavity}.

So, in the sequel $\partial \mathcal U _0$ and $\partial \mathcal U _1$ will be tacitly meant respectively as $\Gamma _ D (\mathcal U _0) \cup O $
and  $\Gamma _ D (\mathcal U _0) \cup O $. 

\smallskip 
With this specification, we now start to discuss what happens when $(\overline x, \overline y)$ belongs to each  of the portions of $\partial \mathcal A$ according to the above decomposition. 
 
\medskip
$\bullet$  If $\overline x\in\partial \mathcal U_0\setminus O$ and 
$\overline  y\in\partial \mathcal U_1\setminus O $, we are done because we have $ \widetilde \C (\overline x, \overline y) = 0$. 
 
 \medskip 
 $\bullet$ 
If $\overline  x\in\partial \mathcal U_0\setminus O$ and 
$\overline  y\in{\rm int}\, \mathcal U_1\setminus O$, or if 
$\overline  x\in{\rm int}\, \mathcal U_0\setminus O$ and 
$\overline  y\in\partial \mathcal U_1\setminus O$,  we reach a contradiction by arguing as done in the proof of Theorem 20 in \cite{C96}, 
 in the case when one of the two points is in the interior and the other one is on the boundary of  the corresponding convex bodies. 
 We refer to that paper for the details of the argument which, similarly as Case 1 above, is purely local. The idea is that, by exploiting Hopf boundary point lemma and the homogeneous Dirichlet condition along $\Gamma _D$,
  the first order variation of $\widetilde \C $ under a small perturbation of $(\overline x, \overline y)$ in inward normal direction is $- \infty$, contradicting the minimality of
 $\widetilde \C$ at  $(\overline x,   \overline y)$.

 \medskip 
 $\bullet$ If $(\overline x , \overline y ) = (O, O)$, we trivially have $ \widetilde \C (\overline x, \overline y) = 0$.

 \medskip  $\bullet$ If $ (\overline x , \overline y )   \in   \big( O \times (\mathcal U_1\setminus O)\big)
\cup
\big((\mathcal U_0\setminus O)\times O \big)$, we reach a contradiction by using Corollary \ref{c:hotspot}. 
Indeed assume to fix the ideas that $O= \overline x$  (the case when $O =  \overline y$ being analogous). We consider $x _\e\in [O,   \overline y ]$ such that 
$x _\e \to O$ as $\e \to 0$. We claim that, for $\e$ small enough,  
  $$\widetilde \C ( x _\e,  \overline y,  t) < \widetilde \C ( O,  \overline  y,  t)\,,$$ 
contradicting the minimality of $\widetilde  \C$ at $ ( \overline x ,   \overline  y,  t)$. Indeed, this amounts to show that 
 $$ \begin{array}{ll}
 &  \mathcal U^{n+1}  (v_{ t} )(  ( 1 - t ) x_\e +   t   \overline  y)  - ( 1-   t)  \,  \mathcal U^{n+1}  (v _0)( x_\e)  -   t  \, \mathcal U^{n+1} (v_1) ( \overline    y)  <
 \\ \noalign{\medskip}
 & 
    \mathcal U^{n+1} (v_{  t} ) (  ( 1 -  t ) O  +   t  \overline   y)  - ( 1-   t)  \,   \mathcal U^{n+1} ( v _0 )  ( O)  -   t  \,  \mathcal U^{n+1} (v_1 )(  \overline  y) \,.
    \end{array}
    $$ 
 This inequality is satisfied  because, respectively by Corollary \ref{c:hotspot} (i)   and Corollary \ref{c:hotspot} (ii), applied on 
 $\mathcal R ^ j ( K_0)$ and on $\mathcal R ^ j ( K _t)$ for some $j \in \{ 0, \dots, n\}$,  
  we have that 
 $$  \begin{cases}
 \mathcal U^{n+1}   ( v _0  ) ( x_\e) = \mathcal U^{n+1}   (v _0 ) ( O)   + o ( \e) &  \\
  \noalign{\medskip} 
 \mathcal U^{n+1}   (   v_{  t} ) (  ( 1 -  t ) x_\e +   t   \overline y) ) =\mathcal U^{n+1}   ( v_{  t} ) (  ( 1 -   t ) O  +   t  \overline   y)  + \gamma \e + o ( \e)\,, \quad \text{ with } \gamma <0\,. &
\end{cases} $$

\medskip   $\bullet$ If $ (\overline x , \overline y )   \in {\rm int}\, \mathcal U_0\times{\rm int}\, \mathcal U_1:
\overline x, \overline  y\neq O,\ |\theta_{\overline  x} -\theta_{\overline  y}|=\pi$  
 
 \smallskip 
 
Notice that, in this case, we have that $O\in [\overline  x, \overline  y]$, with $\overline x, \overline  y\neq O$.   Then we reach a contradiction by using
again Corollary \ref{c:hotspot}. 
Assume that $(1-t ) \overline  x + t \overline  y$ lies in  $[O, \overline  y]$ 
 (the case when it lies in $[O, \overline  x]$ being analogous).  We claim that
  $$ \widetilde \C (O, \overline   y, t) <  \widetilde \C (\overline   x, \overline   y, t)\,,$$
 contradicting the minimality of $\widetilde \C$ at  $ ( \overline  x ,  \overline  y, t)$. Indeed, this amounts to show that 
  $$ \begin{aligned}  
 & \mathcal U^{n+1}   (  v_t )(  ( 1 -t )  O  + t   \overline  y)  - ( 1- t)  \,  \mathcal U^{n+1}   (  v _0  )( O)  - t  \,  \mathcal U^{n+1} \,   ( v_1 )( \overline   y)  < 
 \\ \noalign{\medskip} 
&  \mathcal U^{n+1}   ( v_t  ) (  (1-t )\overline   x + t \overline  y  )  - ( 1- t) \,   \mathcal U^{n+1}   (   v _0 ) (\overline    x )  - t  \,  \mathcal U^{n+1} \,   ( v_1)  (  \overline  y) \,.
 \end{aligned} 
 $$ 
 This inequality is satisfied  because,  using again Corollary \ref{c:hotspot} (i)  on 
 $\mathcal R ^ j ( K_0)$ and on $\mathcal R ^ j ( K _t)$ for some $j \in \{ 0, \dots, n\}$, we have that 
 $$  \begin{cases}
   \mathcal U^{n+1}   ( v _0)  ( O) >  \mathcal U^{n+1}       ( v _0 ) ( \overline   x)   &  \\
  \noalign{\medskip} 
  \mathcal U^{n+1}   ( v_t ) ( (1-t ) \overline  x + t\overline   y   )  >   \mathcal U^{n+1}   ( v_t ) (  ( 1 -t )  O  + t  \overline   y) 
   \,. &
\end{cases} $$  
 
 \smallskip 
 We conclude that $m = 0$, and our proof is achieved. 
   \qed

\bigskip\bigskip 
 \subsection{Proof of  Brunn--Minkowski inequality \eqref{f:BMl} for  the mixed first eigenvalue}

 As mentioned in the Introduction, our approach is inspired by  the one of Theorem 1 in Colesanti's paper \cite{C96}. 
Although the proof below is essentially self-contained, the reader is referred to that paper for some details, 
that we omit in order to focus mainly on the differences arising from the Neumann portion of the boundary.    We state separately two preliminary lemmas. 

\begin{lemma}\label{l:C1} 
Let $u$ be the mixed ground state of a   convex curvilinear sector   $\Om\in \mathcal K _\alpha$,  
with $\alpha \in (0, \frac  \pi 2]$, and  let $v:= - \log u$. 
Let $$X:=  \Big \{p \in \R ^ 2 \  : p \cdot \nu ' < 0 \, , \ p \cdot \nu ''  < 0   \Big \}\,,$$
where $\nu ', \nu ''$are  the unit outward normals  to $\Om$ along the Neumann sides $S', S''$ of $\Om$.
 Then:
\begin{itemize}
\item[(i)]  $\nabla v: \Om \to  X$ is a diffeomorphism of class $C ^ \infty$;

\smallskip 
\item[(ii)] the Fenchel conjugate $v ^* (p) := \sup _ {x \in \Om } \big (  p \cdot x  - v ( x)\big )$ is of class $C ^ \infty ( X)$, and 
$$
\nabla v^*(p) = (\nabla v)^{-1}(p),
\qquad
\nabla ^2 v^*(p) = \bigl(\nabla ^2 v((\nabla v)^{-1}(p))\bigr)^{-1}.
$$
\end{itemize} 
\end{lemma} 
\proof

(i) By standard elliptic regularity, $v$ is smooth in $\Om$. Moreover,   
since $\nabla ^ 2 v$   is positive definite in $\Om$ thanks to Theorem \ref{t:concavity}, the map $\nabla v$ turns out to be injective in $\Om$, 
and hence, by applying the inverse function theorem, we deduce that $\nabla v$ is a diffeomorphism of class $C ^ \infty$  from $\Om$ onto its image  $\nabla v (\Om)$. 
It remains to prove that $\nabla v (\Om) = X$.   
To that aim, we choose a system of coordinates with the origin at the (NN)-vertex of $\Om$. 

\smallskip 
Let us prove first the inclusion $X \subseteq \nabla v (\Om)$.
Let $p \in X$, and set
\begin{equation}\label{f:phip} 
\Phi_p(x) :=   p \cdot x   - v (x).
\end{equation}
Taking into account that $\Omega$ is bounded, that $\Phi_p(x)\to -\infty$ as $x$ approaches $\Gamma_D$, and that $v$ is strictly convex,  we can assert that the
maximum of $\Phi _p$ over $\overline{\Omega}$ is uniquely attained. 
Let $x _p \in {\overline \Om}$ denote the point where such maximum is attained. 
We claim that $x_p \in \Omega$, which gives in particular, by optimality, 
\begin{equation}\label{f:critp} 
\nabla v(x_p) = p\,,
\end{equation}
showing that $p \in \nabla v (\Om)$. 

Since, as we have already observed, $\Phi_p(x)\to -\infty$ as $x$ approaches $\Gamma_D$,  we have $x_p \not \in \Gamma_D$.  On the other hand, 
$x_p$ cannot be a point of $ \Gamma _N$,   because  the assumption that $p \in X$ ensures that the outward  normal derivative of $\Phi _p$ at any such point is strictly negative. Indeed, if $x \in S' \setminus \{ 0 \}$, we have 
$$
\partial_{\nu'} \Phi_p (x)
=
p \cdot \nu'  - \partial_{\nu'} v  (x) 
=
p \cdot \nu'  <  0\,,
$$
and similarly if $x \in S'' \setminus \{ 0 \}$. 
It remains to show that $x _p \neq 0$.  By Lemma \ref{l:expansion} we have $\nabla v(0) = 0$, so that 
$$
\nabla \Phi_p(0) = p \neq 0.
$$
Since $p \in X$, there exists a  vector    $d$ pointing inside $\Omega$ such that $p \cdot d > 0$. Then $0$ cannot be a maximizer because for $t$ small we have  
$$
\Phi_p(td) = \Phi_p(0) + t\, p \cdot d + o(t) > \Phi_p(0)\,. 
$$

\medskip

Let us show the converse inclusion $ \nabla v (\Om)  \subseteq X$,  namely that $\nabla v(x) \in X$ 
for every
$x \in \Omega$. Consider the Neumann side $S'$. Since $\Omega$ is convex, there exists a point $x' \in S'$ such that
$$
x - x' =  - t \nu ' \quad \text{for some } t > 0.
$$
By  the strict convexity of $v$, we have 
$$
(\nabla v(x) - \nabla v(x')) \cdot (x - x') > 0.
$$
 Since $\nabla v ( x') \cdot \nu' = 0$,  we deduce that 
$$
\nabla v(x)  \cdot \nu'  = 
 (\nabla v(x) - \nabla v(x')) \cdot \nu' <  0.
$$
In the analogous way, we see that $\nabla v(x) \cdot \nu '' < 0$, and we conclude that $\nabla v ( x) \in X$. 

\medskip
(ii) As shown above, for every 
 $p \in X$ the function
$\Phi_p$ defined in \eqref{f:phip} admits exactly one maximiser $x_p$, which is characterised by  the equality \eqref{f:critp}. Hence, 
$$v^*(p) = (\nabla v)^{-1}(p) \cdot p - v( (\nabla v)^{-1}(p) ).
$$
Since  the map $\nabla v$ is a smooth diffeomorphism, it follows  that $v ^*$ is of class $C ^ \infty(X)$, and the expressions of $\nabla v ^ * (p)$ and $\nabla^2 v ^ * (p)$
are easily obtained  by differentiating  the above equality with respect to $p$.
\qed

\bigskip 

\begin{lemma}\label{l:C2} Let $\Om\in \mathcal K _\alpha$, for some $\alpha \in (0, \frac  \pi 2]$, and let  $v_i:= - \log u_i$, being  $u_i$ the mixed ground state of $\Om$. For every $z \in \Om _t:= (1-t ) \Om _0 + t \Om _ 1$,  let $w := ( 1-t) v _0 \square t v _1$ denote the infimal convolution
$$w ( z):= \inf \Big \{ ( 1 -t) v _0(x) + t v _ 1 (y) \ :\ x \in \Om _0\, , \ y \in \Om _1 \,,\ ( 1-t ) x + t y = z \Big \}\,.$$ 
Then it holds: 
\begin{itemize}
\item[(i)] for every $z \in \Om _ t$, $w (z)$ is  attained at some pair of points $(x, y)\in  \Om _0 \times \Om _1$; 
\smallskip

\item[(ii)] $w$ is of class $C ^ \infty (\Om _ t)$, and, if $w ( z)$ is attained  at $(x, y)$, it holds
$$\nabla w ( z) = \nabla  v_0   ( x) = \nabla   v_1    ( y) \, , \qquad \nabla^2 w ( z) = [ (1-t) (\nabla  v_0  ( x)) ^ { -1}  +t (\nabla    v_1    ( y) ) ^ { -1}  ] ^ { -1}\,; $$ 

\item[ (iii)] if $\Gamma ^ t _ D$ and $\Gamma ^ t _ N$  denote respectively the Dirichlet and Neumann boundaries of   $\Om _ t$
according to the decomposition adopted in the Introduction,   
it holds
$$w ( z) \to + \infty \ \text{ as } z \to \Gamma ^ t _ D \,, \qquad \partial _\nu  w   = 0 \ \text { on } \Gamma ^ t _N \,.$$ 
  where $\nu$ is the outward unit normal to $\Om _t$.   
 
\end{itemize}

\end{lemma}

\proof  
(i) Let $( x_h, y _h) \in \Om _0 \times \Om _1$, with $( 1-t ) x_h + t y _h=z$, be  a minimizing sequence for $w ( z)$.  Up to a (not relabeled)  subsequence, we have
$( x_h, y _h) \to ( x,  y)  \in \overline \Om _0 \times  \overline \Om _1$.  
 Let us show first  that $x$ (resp.\ $y$)  cannot lie in $\Gamma _ D (\Om _0)$ (resp.\ $\Gamma _ D (\Om _1)$), 
then 
that $x$ cannot lie in $\Gamma _ N (\Om _0)$ (resp.\ $\Gamma _ N (\Om _1)$), and finally that $x$  cannot coincide with the (NN) vertex of  $\Om _0$
(resp.\ $\Om _1$).

The possibility that $ x  \in \Gamma _ D (\Om _0)$ or  $ y  \in \Gamma _ D (\Om _1)$ is  immediately ruled out by the fact that $v_i $ tends to $+ \infty$ towards $\Gamma _ D (\Om _i)$.  
In order to exclude that $ x \in \Gamma _ N  (\Om _0)$ or $ y \in \Gamma _ N  (\Om _1)$, let us assume for definiteness that $ x$ belongs to  a Neumann side $S'$ of $\Om _0$,   the other case being analogous. 
Denoting by $\nu'$ be the unit outer normal to $\Om _0$ along $S'$,  
we observe that 
the point $y$ cannot lie on the corresponding Neumann side of $\Om _1$  (i.e., the side of $\Om _1$ orthogonal to $\nu'$), otherwise 
$z=(1-t) x+t y$ would belong to the corresponding Neumann side of $\Omega_t$  (i.e., the side of $\Om _t$ orthogonal to $\nu'$), against the assumption that 
$z$ belongs to $\Omega_t$.  Then we can consider, for $\varepsilon>0$ small, the admissible variation 
\begin{equation}\label{f:avar} 
x_\varepsilon= x- \varepsilon \nu'\,, 
\qquad
y_\varepsilon= y+ \frac{1-t}{t}\varepsilon\nu' , 
\end{equation} 
which contradicts the minimality at $( x,  y)$, since
$$
\begin{aligned} 
\frac{d}{d \varepsilon} \Big (  (1-t)v_0(x_\varepsilon)+t v_1(y_\varepsilon) \Big ) \Big | _{\varepsilon = 0 } 
& = - (1-t)\big(\nabla v_0( x)-\nabla v_1( y)\big)\cdot \nu' 
\\
& = 
(1-t)\big(\nabla v_1(y)\big)\cdot \nu '   <  0\,,
\end{aligned} 
$$ 
where the last inequality follows from Lemma \ref{l:C1}. 

Finally, the  possibility that $ x$ or $ y$  agrees with the (NN)-vertex of the corresponding   curvilinear sector   can be excluded in a similar way: for instance, if $ x$ coincides with the (NN)-vertex of $\Om_0$, we can consider a perturbation as in \eqref{f:avar}, with $- \nu'$ replaced by a vector $h$ lying in the open cone $X$ defined in Lemma \ref{l:C1}.    The minimality at $( x, y)$ is contradicted  by the same computation as above, taking into account that $\nabla v_0 (  x) = 0.$

\smallskip
(ii) Since $w ^ * = (1-t) v_0 ^* + t v _ 1 ^*$ (see \cite[Theorem 16.4]{Rock}), we infer from Lemma \ref{l:C1} (ii) that $w^*$ is of class $C ^ \infty (X)$, with $\nabla ^ 2 w ^*$ positive definite in $X$. Thus, $\nabla w ^*$ is a diffeomorphism of class $C ^ \infty$ from $X$ onto $\nabla w ^* (X)$. By arguing in  analogous way as in the proof of Lemma \ref{l:C1}, one gets the equality $\nabla w ^* (X) = \Om _t$, and consequently that $w $ (which agrees with $w^{**}$, see  \cite[Theorem 16.4]{Rock} ) is of class $C ^ \infty (\Om _t)$. 
   Then the equalities $\nabla w ( z) = \nabla w ( x) = \nabla w ( y)$ and $(\nabla^2 w ( z))  ^ { -1}  =  (1-t) (\nabla w ( x)) ^ { -1}  +t (\nabla w ( y) ) ^ { -1}  $ follow by  the same arguments as in the proof of Theorem 1 in \cite{C96}.

\smallskip
(iii)  Let us prove first that, if $\{ z _h \}$ is a sequence of points in $\Om _t$ such that
$z _h \to z \in \Gamma^t _ D$, we have
$w ( z_h) \to + \infty$.  
 Let $(x_h, y _h) \in \Om _ 0 \times \Om _1$ be such that 
$z_h=(1-t) x_h+t y_h$  and $w ( z_h) = ( 1- t) v _0 ( x _h) + t v _ 1 ( y _h)$.   Up to subsequences, let $(x_h, y _h) \to ( x,  y)$. 
Thus, by passing to the limit, we obtain 
$z = (1-t)  x + t  y$. Since  $z\in \Gamma ^t_ D$, it cannot be obtained as a Minkowski combination of points which are not 
in the Dirichlet portions of $\Om _0$ and $\Om _1$.  

Therefore, $v _0 ( x _h)$ and $v _ 1 ( y _h)$ tend to $+ \infty$, and hence the same holds for $w ( z_h)$. 
 
Let us now prove that the normal derivative of $w$ vanishes on $\Gamma ^ t _N$. 
Let $z$ be a point of $\Gamma ^ t_N$ and let $ \{ z _h \} \subset \Omega_t$ be a sequence converging to $z$.  For each $h$, let $(x_h,y_h)$ be a minimizing pair in the definition of $w(z_h)$, with
$z_h=(1-t)x_h+t y_h$. Since the three domains have the same Neumann directions, we have 
$$
d_t(z_h)
=
(1-t)d_0(x_h)+t d_1(y_h),
$$
where $d_t$ denotes the distance from the boundary of $\Om _t$. Since $z_h\to z \in \Gamma ^ t _N$, we have that  $d_t (z_h)\to 0$, and therefore also 
$d_0(x_h)\to 0$ and 
$d_1(y_h)\to 0$, meaning that 
$x_h$ and $y_h$ approach the corresponding Neumann sides of $\Om _0$ and $ \Om _1$. 
By local boundary regularity up to flat Neumann sides, away from the vertices, we have $
\nabla v_j(x)\cdot \nu  \to 0$ when $x $ approaches $\Gamma _N$ from inside the domain, locally away from the vertices. 
Applying this to $x_h$ and $y_h$, we get
$$
\nabla v_0(x_h)\cdot \nu \to 0,
\qquad
\nabla v_1(y_h)\cdot \nu \to 0.
$$
Since
$$
\nabla w(z_h)=\nabla v_0(x_h)=\nabla v_1(y_h),
$$
it follows that
$\nabla w(z_h)\cdot \nu \to 0$, namely the 
interior trace of the normal derivative of $w$ on $\Gamma ^ t _N$ exists and satisfies
$
\partial_{\nu} w(z)=0.
$ \qed

\bigskip
\bigskip
We are now in a position to give the proof of  inequality   \eqref{f:BMl}.   We consider the function defined for every  $z \in \Om _ t$ by
$$\overline u (z): = e ^ { - w ( z)} \,,$$
where $w (z):= ( 1-t) v _0 \square t v _1$ is the infimal convolution of $v _ 0:= - \log u _0$ and $v _1 := - \log u _ 1$, 
defined according to Lemma \ref{l:C2}.

By using the second equality in Lemma  \ref{l:C2} (ii), and the convexity of the map $A \to {\rm tr} ( A^ {-1})$, we deduce that
$$\Delta w(z)  \leq  ( 1- t) \Delta v_0 (x) + t \Delta v _ 1 (y) \,.$$
Then, by exploiting the PDEs satisfied by $v _i$ in $\Om _i$, which reads
$$\Delta v _ i = \lambda (\Om _i) + |\nabla v _ i | ^ 2 \,,$$
and the first equality in Lemma \ref{l:C2} (ii), we obtain   
$$\Delta w(z)  \leq  ( 1- t) \lambda (\Om _0) + t \lambda (\Om _ 1) + |\nabla w ( z)| ^ 2\,.$$ 

It follows that $\overline u$ satisfies 
$$\Delta \overline u \geq -  [ ( 1- t) \lambda (\Om _0) + t \lambda (\Om _ 1) ] \overline u \qquad \text{ in } \Om _ t \,.$$ 
We multiply the above inequality by $\overline u$, and we integrate over $\Om _t$. Taking into account that, by Lemma \ref{l:C2} (iii)  it holds
$$\overline u = 0 \ \text{ on } \Gamma ^ t _ D \,,  \qquad \partial _\nu \overline u  = 0 \ \text{ on } \Gamma ^ t _N \,,$$ 
we obtain
$$\int _{ \Om _ t}  |\nabla \overline u | ^ 2 \leq  [ ( 1- t) \lambda (\Om _0) + t \lambda (\Om _ 1) ]  \int _{ \Om _ t}  |\overline u | ^ 2 \,.$$ 
Then   \eqref{f:BMl}   follows from the variational characterization of $\lambda (\Om _ t)$. \qed

 \section*{Acknowledgments} D.B.  would like to thank the Isaac Newton Institute for Mathematical Sciences, Cambridge and the Simons Foundation for support and hospitality during the programme {\it Geometric Spectral Theory} where part of the  work on this paper was undertaken.   This work was supported by EPSRC grant no EP/Z000580/1. D.B.  was 
 also partially supported by the ANR project STOIQUES financed by the French Agence Nationale de la Recherche (ANR-24-CE40-2216).


\begin{thebibliography}{10}

\bibitem{AR23}
{N.} Aldeghi and {J.} Rohleder, \emph{Inequalities between the lowest
  eigenvalues of {L}aplacians with mixed boundary conditions}, J. Math. Anal.
  Appl. \textbf{524} (2023), no.~1, Paper No. 127078, 16.

\bibitem{AR25}
{N.} Aldeghi and {J.} Rohleder, \emph{On the first eigenvalue and eigenfunction
  of the {L}aplacian with mixed boundary conditions}, J. Differential Equations
  \textbf{427} (2025), 689--718.

\bibitem{ALL}
{O.} Alvarez, {J.-M.} Lasry, and {P.-L.} Lions, \emph{Convex viscosity
  solutions and state constraints}, J. Math. Pures Appl. (9) \textbf{76}
  (1997), no.~3, 265--288.

\bibitem{AnClHa}
{B.} Andrews, {J.} Clutterbuck, and {D.} Hauer, \emph{Non-concavity of the
  {R}obin ground state}, Camb. J. Math. \textbf{8} (2020), no.~2, 243--310.

\bibitem{BB99}
{R.} Ba\~nuelos and {K.} Burdzy, \emph{On the ``hot spots'' conjecture of {J}.\
  {R}auch}, J. Funct. Anal. \textbf{164} (1999), no.~1, 1--33.

\bibitem{BerPac89}
{H.} Berestycki and {F.} Pacella, \emph{Symmetry properties for positive
  solutions of elliptic equations with mixed boundary conditions}, J. Funct.
  Anal. \textbf{87} (1989), no.~1, 177--211.

\bibitem{B85}
{C.} Borell, \emph{Greenian potentials and concavity}, Math. Ann. \textbf{272}
  (1985), no.~1, 155--160.

\bibitem{BL76}
{H.J.} Brascamp and {E.H.} Lieb, \emph{On extensions of the {B}runn-{M}inkowski
  and {P}r\'{e}kopa-{L}eindler theorems, including inequalities for log concave
  functions, and with an application to the diffusion equation}, J. Functional
  Analysis \textbf{22} (1976), no.~4, 366--389.

\bibitem{CafFri}
{L.A.} Caffarelli and {A.} Friedman, \emph{Convexity of solutions of semilinear
  elliptic equations}, Duke Math. J. \textbf{52} (1985), no.~2, 431--456.

\bibitem{CGY26}
{H.} Chen, {C.} Gui, and {R.} Yao, \emph{Uniqueness of critical points of the
  second {N}eumann eigenfunctions on triangles}, Invent. Math. \textbf{244}
  (2026), no.~1, 299--353.

\bibitem{CY18}
{H.} Chen and {R.} Yao, \emph{Symmetry and monotonicity of positive solution of
  elliptic equation with mixed boundary condition in a spherical cone}, J.
  Math. Anal. Appl. \textbf{461} (2018), no.~1, 641--656.

\bibitem{C96}
A.~Colesanti, \emph{Brunn-{M}inkowski inequalities for variational functionals
  and related problems}, Adv. Math. \textbf{194} (2005), no.~1, 105--140.

\bibitem{CFrobin}
{G.} Crasta and {I.} Fragal\`a, \emph{Concavity properties of solutions to
  {R}obin problems}, Camb. J. Math. \textbf{9} (2021), no.~1, 177--212.

\bibitem{GuanMa}
{P.} Guan and {X.} Ma, \emph{Convex solutions of fully nonlinear elliptic
  equations in classical differential geometry}, Geometric evolution equations,
  Contemp. Math., vol. 367, Amer. Math. Soc., Providence, RI, 2005,
  pp.~115--127.

\bibitem{Hat24}
{L.} Hatcher, \emph{First mixed Laplace eigenfunctions with no hot spots},
  Proc. Amer. Math. Soc. \textbf{152} (2024), no.~12, 5191--5205.

\bibitem{JM20}
{C.} Judge and {S.} Mondal, \emph{Euclidean triangles have no hot spots}, Ann.
  of Math. (2) \textbf{191} (2020), no.~1, 167--211.

\bibitem{Kbook}
{B.} Kawohl, Rearrangements and convexity of level sets in {PDE}, Lecture Notes
  in Mathematics, vol. 1150, Springer-Verlag, Berlin, 1985.

\bibitem{Kenn}
{A.U.} Kennington, \emph{Power concavity and boundary value problems}, Indiana
  Univ. Math. J. \textbf{34} (1985), no.~3, 687--704.

\bibitem{K83}
{N.J.} Korevaar, \emph{Convex solutions to nonlinear elliptic and parabolic
  boundary value problems}, Indiana Univ. Math. J. \textbf{32} (1983), no.~4,
  603--614.

\bibitem{KorLew}
{N.J.} Korevaar and {J.L.} Lewis, \emph{Convex solutions of certain elliptic
  equations have constant rank {H}essians}, Arch. Rational Mech. Anal.
  \textbf{97} (1987), no.~1, 19--32. \MR{856307}

\bibitem{Leb72}
{N.N.} Lebedev, Special functions and their applications, revised ed., Dover
  Publications, Inc., New York, 1972, Unabridged and corrected republication.

\bibitem{PL1}
{L.} Leindler, \emph{On a certain converse of {H}\"older's inequality. {II}},
  Acta Sci. Math. (Szeged) \textbf{33} (1972), no.~3-4, 217--223.

\bibitem{LY26}
{R.} Li and {R.} Yao, \emph{Monotonicity of positive solutions to semilinear
  elliptic equations with mixed boundary conditions in triangles}, J. Funct.
  Anal. \textbf{290} (2026), no.~12, Paper No. 111448.

\bibitem{LR17}
{V.} Lotoreichik and {J.} Rohleder, \emph{Eigenvalue inequalities for the
  {L}aplacian with mixed boundary conditions}, J. Differential Equations
  \textbf{263} (2017), no.~1, 491--508.

\bibitem{maklim}
{L. G.} Makar-Limanov, \emph{The solution of the {D}irichlet problem for the
  equation {$\Delta u=-1$} in a convex region}, Mat. Zametki \textbf{9} (1971),
  89--92.

\bibitem{PPR24}
{F.} Pacella, {G.} Poggesi, and {A.} Roncoroni, \emph{Optimal quantitative
  stability for a {S}errin-type problem in convex cones}, Math. Z. \textbf{307}
  (2024), no.~4, Paper No. 79, 19.

\bibitem{PT20}
{F.} Pacella and {G.} Tralli, \emph{Overdetermined problems and constant mean
  curvature surfaces in cones}, Rev. Mat. Iberoam. \textbf{36} (2020), no.~3,
  841--867.

\bibitem{PT21}
{F.} Pacella and {G.} Tralli, \emph{Isoperimetric cones and minimal solutions
  of partial overdetermined problems}, Publ. Mat. \textbf{65} (2021), no.~1,
  61--81.

\bibitem{PM}
Polymath, \emph{{Polymath Research Thread 5: the hot spot conjecture}},
  2012-13, available at
  https://polymathprojects.org/2013/08/09/polymath7-research-thread-5-the-hot-spots-conjecture/.

\bibitem{PL2}
{A.} Pr\'ekopa, \emph{On logarithmic concave measures and functions}, Acta Sci.
  Math. (Szeged) \textbf{34} (1973), 335--343.

\bibitem{Rock}
{R.T.} Rockafellar, Convex Analysis, Princeton Univ.\ Press, Princeton, NJ,
  1970.

\bibitem{S16}
{B.} Siudeja, \emph{On mixed {D}irichlet-{N}eumann eigenvalues of triangles},
  Proc. Amer. Math. Soc. \textbf{144} (2016), no.~6, 2479--2493.

\bibitem{YCG21}
{R.} Yao, {H.} Chen, and {C.} Gui, \emph{Symmetry of positive solutions of
  elliptic equations with mixed boundary conditions in a super-spherical
  sector}, Calc. Var. Partial Differential Equations \textbf{60} (2021), no.~4,
  Paper No. 130, 25.

\bibitem{YCL18}
{R.} Yao, {H.} Chen, and {Y.} Li, \emph{Symmetry and monotonicity of positive
  solutions of elliptic equations with mixed boundary conditions in a
  super-spherical cone}, Calc. Var. Partial Differential Equations \textbf{57}
  (2018), no.~6, Paper No. 154, 28.

\bibitem{Zhu01}
{M.} Zhu, \emph{Symmetry properties for positive solutions to some elliptic
  equations in sector domains with large amplitude}, J. Math. Anal. Appl.
  \textbf{261} (2001), no.~2, 733--740.

\end{thebibliography}
\end{document}